\documentclass{amsart}

\usepackage{amssymb}
\usepackage{mathrsfs}
\usepackage[utf8]{inputenc}
\usepackage{enumitem}
\usepackage{hyperref}
\usepackage{verbatim}
\usepackage{xcolor}
\usepackage{tikz}
\usetikzlibrary{calc,arrows,intersections}

\newcounter{propert}
\makeatletter
\newcommand\myitem[1][]{\item[#1]\refstepcounter{propert}\def\@currentlabel{#1}}
\makeatother

\newlist{thmlist}{enumerate}{1}
\setlist[thmlist]{label=(\arabic{thmlisti}), ref=\thethm./(\arabic{thmlisti}),noitemsep}

\usepackage[capitalize]{cleveref}
\Crefname{thm}{Theorem}{Theorems}
\Crefname{lem}{Lemma}{Lemmas}

\numberwithin{equation}{section}

\newtheorem{theorem}{Theorem}[section]
\newtheorem{thm}[theorem]{Theorem}
\newtheorem{lem}[theorem]{Lemma}
\newtheorem{prop}[theorem]{Proposition}
\newtheorem{cor}[theorem]{Corollary}

\newtheorem{conj}[theorem]{Conjecture}

\theoremstyle{definition}
\newtheorem{defin}[theorem]{Definition}

\theoremstyle{remark}
\newtheorem{rem}[theorem]{Remark}

\newcommand{\FF}[1]{\mathbb F_{#1}}
\newcommand{\ZZ}{\mathbb Z}
\newcommand{\NN}{\mathbb N}

\newcommand{\diam}{\operatorname{diam}}
\newcommand{\Tr}{\operatorname{Tr}}

\newcommand{\mt}{\mathcal T}
\newcommand{\mf}{\mathcal F}
\newcommand{\E}{E(\mt)}

\newcommand{\Cay}{\operatorname{Cay}}

\newcommand{\fl}{\langle}
\newcommand{\fr}{\rangle}
\newcommand{\lGamma}{\widetilde{\Gamma}}
\newcommand{\SL}{\mathrm{SL}}
\newcommand{\SU}{\mathrm{SU}}
\newcommand{\Sp}{\mathrm{Sp}}
\newcommand{\PSL}{\mathrm{PSL}}

\newcommand{\GL}{\mathrm{GL}}
\newcommand{\GU}{\mathrm{GU}}
\newcommand{\PGL}{\mathrm{PGL}}

\newcommand{\tikzshortencycle}{
  \begin{tikzpicture}[scale=0.8,>=stealth]
    \path [name path=kor]     (0,0) circle (4cm);
    \path [name path=r1line]  (0,0) -- (140:5cm);
    \path [name path=r2line]  (0,0) -- (170:5cm);
    \path [name path=ri-1line](0,0) -- (240:5cm);
    \path [name path=riline]  (0,0) -- (270:5cm);
    \path [name path=ri+1line](0,0) -- (300:5cm);
    \path [name path=rk-2line](0,0) -- ( 40:5cm);
    \path [name path=rk-1line](0,0) -- ( 70:5cm);
    \path [name path=rkline]  (0,0) -- (110:5cm);

      \path [name intersections={of=kor and r1line,   by=R1}];
      \path [name intersections={of=kor and r2line,   by=R2}];
      \path [name intersections={of=kor and ri-1line, by=Ri-1}];
      \path [name intersections={of=kor and riline,   by=Ri}];
      \path [name intersections={of=kor and ri+1line, by=Ri+1}];
      \path [name intersections={of=kor and rk-2line, by=Rk-2}];
      \path [name intersections={of=kor and rk-1line, by=Rk-1}];
      \path [name intersections={of=kor and rkline,   by=Rk}];

      \draw[very thick,dotted] (R2)   arc (170:240:4cm);
      \draw[very thick,dotted] (Ri+1) arc (300:400:4cm);

      \node[fill=white] (R1b)   at (R1)   {$r_1$};
      \node[fill=white] (R2b)   at (R2)   {$r_2$};
      \node[fill=white] (Ri-1b) at (Ri-1) {$r_{i-1}$};
      \node[fill=white] (Rib)   at (Ri)   {$r_i$};
      \node[fill=white] (Ri+1b) at (Ri+1) {$r_{i+1}$};
      \node[fill=white] (Rk-2b) at (Rk-2) {$r_{k-2}$};
      \node[fill=white] (Rk-1b) at (Rk-1) {$r_{k-1}$};
      \node[fill=white] (Rkb)   at (Rk)   {$r_k$};

      \path[->]
      (R1b)   edge [bend right]node[above,sloped]{$\scriptstyle{a_1}$}   (R2b)
      (Ri-1b) edge [bend right]node[below,sloped]{$\scriptstyle{a_{i-1}}$}(Rib)
      (Rib)   edge [bend right]node[below,sloped]{$\scriptstyle{a_i}$}  (Ri+1b)
      (Rk-2b) edge[bend right]node[above,sloped]{$\scriptstyle{a_{k-2}}$}(Rk-1b)
      (Rk-1b) edge [bend right]node[above,sloped]{$\scriptstyle{a_{k-1}}$}(Rkb)
      (Rkb)   edge [bend right]node[above,sloped]{$\scriptstyle{a_k}$}    (R1b)
      (Rk-2b) edge [bend left]node[below,sloped]{$\scriptstyle{b}$}      (Rkb)
      (Rk-1b) edge [bend left]node[below,sloped]{$\scriptstyle{c}$}     (R1b)
      (Rkb)   edge [bend right]node[above,sloped]{$\scriptstyle{d}$}     (Rk-1b)
      ;
  \end{tikzpicture}
}
\newcommand{\tikzthreecycle}{
  \begin{tikzpicture}[scale=0.4,>=stealth]
    \path [name path=kor]     (0,0) circle (4cm);
    \path [name path=r1line]  (0,0) -- (90:5cm);
    \path [name path=r2line]  (0,0) -- (210:5cm);
    \path [name path=r3line]  (0,0) -- (330:5cm);

    \path [name intersections={of=kor and r1line,   by=R1}];
    \path [name intersections={of=kor and r2line,   by=R2}];
    \path [name intersections={of=kor and r3line,   by=R3}];

    \node[fill=white] (R1b)   at (R1)   {$r_1$};
    \node[fill=white] (R2b)   at (R2)   {$r_2$};
    \node[fill=white] (R3b)   at (R3)   {$r_3$};

    \draw[->] (R1b) to [out=225,in=75]  node[left]{$\scriptstyle{a_1}$} (R2b);
    \draw[->] (R2b) to [out=-15,in=195] node[below]{$\scriptstyle{a_2}$}(R3b);
    \draw[->] (R3b) to [out=105,in=315] node[right]{$\scriptstyle{a_3}$}(R1b);
    \draw[->] (R1b) to [out=285,in=135] node[right]{$\scriptstyle{b}$}  (R3b);
    \draw[->] (R3b) to [out=165,in=15 ] node[above]{$\scriptstyle{d}$}  (R2b);
    \draw[->] (R2b) to [out=45 ,in=255] node[left]{$\scriptstyle{c}$}   (R1b);
  \end{tikzpicture}
}

\newcommand{\tikzhatszog}{
  \begin{tikzpicture}[scale=1.2,>=latex]
    \node (Rj)   at ( 0, 2) {$r_j$};
    \node (Rj-1) at (-2.8, 1) {$r_{j-1}$};
    \node (Xl)   at (-1.2, 1) {$x_l$};
    \node (Y1)   at ( 1.2, 1) {$y_1$};
    \node (Rj+1) at ( 2.8, 1) {$r_{j+1}$};
    \node (Ri+1) at (-2.8,-1) {$r_{i+1}$};
    \node (X1)   at (-1.2,-1) {$x_1$};
    \node (Ym)   at ( 1.2,-1) {$y_m$};
    \node (Ri-1) at ( 2.8,-1) {$r_{i-1}$};
    \node (Ri)   at ( 0,-2) {$r_i$};

    \draw[->, thick] (Rj-1) to [out=35 ,in=185] (Rj);
    \draw[->, thick] (Xl)   to [out=60 ,in=200] (Rj);
    \draw[->, thick] (Rj)   to [out=-20,in=120] (Y1);
    \draw[->, thick] (Rj)   to [out=-5 ,in=145] (Rj+1);

    \draw[->, thick, dashed] (Ri+1) to [out=110,in=250] (Rj-1);
    \draw[->, thick, dashed] (X1)   to [out=105,in=255] (Xl);
    \draw[->, thick, dashed] (Y1)   to [out=285,in= 75] (Ym);
    \draw[->, thick, dashed] (Rj+1) to [out=290,in= 70] (Ri-1);

    \draw[->, thick] (Ri)   to [out=175,in=-35] (Ri+1);
    \draw[->, thick] (Ri)   to [out=160,in=-60] (X1);
    \draw[->, thick] (Ym)   to [out=240,in=20 ] (Ri);
    \draw[->, thick] (Ri-1) to [out=215,in=5  ] (Ri);
  \end{tikzpicture}
}
\newcommand{\tikznewcycleunitary}{
  \begin{tikzpicture}[scale=0.7,>=stealth]
    \begin{scope}[shift={(-4,0)}]
      \coordinate (T) at (90:2.5);
      \coordinate (R1) at (50:2.5);
      \coordinate (R2) at (10:2.5);
      \coordinate (Rk-1) at (170:2.5);
      \coordinate (Rk) at (130:2.5);
      \coordinate (S1) at (90:4);

      \draw[very thick,dotted] (Rk-1) arc (170:370:2.5cm) (R2);

      \node[fill=white] (Tb) at (T)  {$t$};
      \node[fill=white] (R1b) at (R1) {$r_1$};
      \node[fill=white] (R2b) at (R2) {$r_2$};
      \node[fill=white] (Rk-1b) at (Rk-1) {$r_{k-1}$};
      \node[fill=white] (Rkb) at (Rk) {$r_k$};
      \node[fill=white] (S1b) at (S1) {$s_1$};

      \draw[<->] (Tb) -- (R1b);
      \draw[<->] (R1b) -- (R2b);
      \draw[<->] (Rk-1b) -- (Rkb);
      \draw[<->] (Rkb) -- (Tb);
      \draw[<->] (Tb) -- (S1b);
      \draw[<->] (S1b) -- (R1b);
      \draw[<->] (S1b) -- (Rkb);
    \end{scope}

    \begin{scope}[shift={(4,0)}]
      \coordinate (T) at (90:2.5);
      \coordinate (R1) at (50:2.5);
      \coordinate (R2) at (10:2.5);
      \coordinate (Rk-1) at (170:2.5);
      \coordinate (Rk) at (130:2.5);
      \coordinate (S1) at (70:4);
      \coordinate (S2) at (110:4);

      \draw[very thick,dotted] (Rk-1) arc (170:370:2.5cm) (R2);

      \node[fill=white] (Tb) at (T)  {$t$};
      \node[fill=white] (R1b) at (R1) {$r_1$};
      \node[fill=white] (R2b) at (R2) {$r_2$};
      \node[fill=white] (Rk-1b) at (Rk-1) {$r_{k-1}$};
      \node[fill=white] (Rkb) at (Rk) {$r_k$};
      \node[fill=white] (S1b) at (S1) {$s_1$};
      \node[fill=white] (S2b) at (S2) {$s_2$};

      \draw[<->] (Tb) -- (R1b);
      \draw[<->] (R1b) -- (R2b);
      \draw[<->] (Rk-1b) -- (Rkb);
      \draw[<->] (Rkb) -- (Tb);
      \draw[<->] (Tb) -- (S1b);
      \draw[<->] (Tb) -- (S2b);
      \draw[<->] (S2b) -- (S1b);
      \draw[<->] (S1b) -- (R1b);
      \draw[<->] (S2b) -- (Rkb);
    \end{scope}
  \end{tikzpicture}
}

\begin{document}
\title[On the diameter of Cayley graphs of classical groups]{On the
  diameter of Cayley graphs of classical groups with generating sets
  containing a transvection}

\author{Martino Garonzi}
\address{Universidade de Bras\'ilia, Departamento de matem\'atica. Campus
  Universit\'ario Darcy Ribeiro. Brasilia - DF. 70910--900, Brazil\newline
  ORCID: \url{https://orcid.org/0000-0003-0041-3131}}
\email{mgaronzi@gmail.com}

\author{Zolt\'an  Halasi}
\address{
  E\"otv\"os Lor\'and University, P\'azm\'any P\'eter s\'et\'any 1/c, H-1117,
  Budapest, Hungary \and Alfr\'ed R\'enyi Institute of Mathematics,
  Re\'altanoda utca 13-15, H-1053, Budapest, Hungary\newline
  ORCID: \url{https://orcid.org/0000-0002-1305-5380}
}
\email{halasi.zoltan@renyi.hu}

\author{G\'abor Somlai}
\address{
  E\"otv\"os Lor\'and University, P\'azm\'any P\'eter s\'et\'any 1/c, H-1117,
  Budapest, Hungary \and Alfr\'ed R\'enyi Institute of Mathematics,
  Re\'altanoda utca 13-15, H-1053, Budapest, Hungary\newline
  ORCID: \url{https://orcid.org/0000-0001-5761-7579}
  }
\email{gabor.somlai@ttk.elte.hu}

\thanks{The first author acknowledges the support of Funda\c{c}\~{a}o
  de Apoio \`a Pesquisa do Distrito Federal (FAPDF) - demanda
  espont\^{a}nea 03/2016, and of Conselho Nacional de Desenvolvimento
  Cient\'ifico e Tecnol\'ogico (CNPq) - Grant numbers 302134/2018-2,
  422202/2018-5.\\
  \indent The work of the second and third authors on the project
  leading to this application has received funding from the European
  Research Council (ERC) under the European Union's Horizon 2020
  research and innovation programme (grant agreement No. 741420).
  Their work was supported by the National
  Research, Development and Innovation Office (NKFIH) Grant
  No.~K138596.\\
  \indent The third author was also partially supported by the
  J\'anos Bolyai Research Fellowship and by the New National
  Excellence Program under the grant number UNKP-20-5-ELTE-231.}

\date{\today}

\subjclass[2010]{}
\begin{abstract}
  A well-known conjecture of Babai states that if $G$ is any finite
  simple group and $X$ is a generating set for $G$, then the diameter
  of the Cayley graph $\Cay(G,X)$ is bounded by $\log|G|^c$ for some
  universal constant $c$. In this paper, we prove such a bound for
  $\Cay(G,X)$ for $G=\PSL(n,q),\ PSp(n,q)$ or $PSU(n,q)$ where $q$ is
  odd, under the assumptions that $X$ contains a transvection and
  $q\neq 9$ or $81$.
\end{abstract}
\maketitle
\markleft{M.~GARONZI ET AL.}
\section{Introduction}

Given a finite group $G$ and a set $X$ of generators of $G$, the
associated (undirected)
Cayley graph $\Cay(G,X)$ is defined to have vertex set $G$ and
edge set $\{ \{g,gx\} : g \in G, \ x \in X \}$. The diameter
of $\Cay(G,X)$ equals the maximum over $g \in G$ of
the length of a shortest expression of $g$ as a product of generators
in $X$ and their inverses.
The maximum of $\diam(\Cay(G,X))$, as
$X$ runs over all possible generating sets of $G$, is denoted by
$\diam(G)$.

In 1992 Babai \cite{BS} proposed the following conjecture.
\begin{conj}
  If $G$ is a non-Abelian finite simple group, then
  $\diam(G)\leq (\log|G|)^c$ for some absolute constant $c$.
\end{conj}
In the same paper Babai and Seress established the following bound:
\[  diam(A_n) < exp(\sqrt{n ~ln (n)}(1 + o(1))).\]
Further, a similar bound holds for arbitrary permutation groups of degree $n$.
The first infinite series of finite simple groups for which the
conjecture was proved by Helfgott \cite{Helfgott} is $PSL(2,p)$, where
$p$ is a prime.

For simple groups of Lie type of bounded rank, Babai's conjecture is
completely solved. Note that this covers the case of exceptional
simple groups of Lie type.  It is an easy consequence of the following
``product theorem'':
\begin{thm}[Pyber--Szab\'o \cite{Pyber-Szabo},
  Breuillard--Green--Tao \cite{Breuillard-etal}]
  \label{thm:product}
  For any positive integer $r$, there is an
  $\varepsilon = \varepsilon(r)>0$ such that if $G$ is any finite
  simple group of Lie type of rank $r$ and $X$ is a generating set of
  $G$, then either $|X^3|>|X|^{1+\varepsilon}$ or $X^3 = G$.
\end{thm}
As a consequence of this deep result one gets
a strong form of Babai's conjecture.
\begin{cor}\label{cor:Babai-strong}
  If $G$ is a finite simple group of Lie type of bounded
  rank $r$, then for any generating set $X$ of $G$
  \[
    \diam(\Cay(G,X))=O\Big(\frac{\log|G|}{\log|X|}\Big)^c,
  \]
  where $c$ depends only on $r$.
\end{cor}
\begin{proof}
  Let $\varepsilon=\varepsilon(r)$ as in \cref{thm:product}
  and let $k$ be the smallest integer satisfying
  $|X|^{(1+\varepsilon)^k}>|G|$.
  Assuming that $X^{(3^k)}\neq G$, by a repeated use of
  \cref{thm:product} we get that
  \[
    |X^{(3^k)}|>|X|^{(1+\varepsilon)^k}>|G|,
  \]
  a contradiction. So, $\diam(\Cay(G,X))\leq 3^k$. On the other hand,
  by choosing $c=\log_{1+\varepsilon} 3$,
  \[
    |X|^{(1+\varepsilon)^k}>|G|\iff
    (1+\varepsilon)^k>\frac{\log|G|}{\log|X|}\iff
    3^k>\Big(\frac{\log|G|}{\log|X|}\Big)^c,
  \]
  that is, $k$ is the smallest integer satisfying
  $3^k>\Big(\frac{\log|G|}{\log|X|}\Big)^c$. Hence
  \[
    \diam(\Cay(G,X))\leq 3^k\leq 3\Big(\frac{\log|G|}{\log|X|}\Big)^c.
  \]
	This concludes the proof.
\end{proof}

In \cite{HZ}, the second author proved a suitable bound for
$\diam(\Cay(G,X))$ when $G=\SL(n,p)$, $p$ is a prime and $X$ is a
generating set for $G$ containing a transvection.  The general case of the conjecture is open. The results obtained in \cite{by},\cite{hmpq} for classical groups of unbounded rank are exponential in $q$. More
precisely, it was proved in \cite{hmpq} that
\[ diam(G) \le q^{O(n \log(n))^2} .\]
On the other hand, for large enough $q$, this was strenghtened by Bajpai, Dona and Helfgott \cite{BDH} who proved that if $G$ is a classical Chevalley group of rank $n$ defined over the field $\mathbb{F}_q$, then $$ diam(G) \le (\log(|G|))^{1947n^4 \log(2n)}.$$

The best available bound for permutation groups was proved by Helfgott
and Seress \cite{hs} that
\[diam (Cay(A_n, S)) = exp~ O((\log n)^4 \log \log n)\] if $S$ is a
generating set of $A_n$.
\begin{rem}
  In many cases, a simple group $G$ is given as the image of a
  quasisimple group $\widetilde{G}$ under a surjective homomorphism
  $\tau: \widetilde{G}\to G$.  Now, if $X$ is any generating set for
  $G$, then $\widetilde{X}=\tau^{-1}(X)$ is a generating set for
  $\widetilde{G}$ satisfying $\diam(\Cay(\widetilde{G},\widetilde{X}))
  = \diam(\Cay(G,X))$.  Since we also have $|\widetilde{G}|\le
  |G|^{O(1)}$, a positive answer to Babai's conjecture for
  $\widetilde{G}$ implies the positive answer to Babai's conjecture
  for $G$.
\end{rem}
The main result of this paper is the following.
\begin{thm}\label{thm:main}
  Let $V$ be an $n$-dimensional vector space over the finite field
  $\FF q$ where $q$ is odd and $G$ is one of $\SL(V),\ \Sp(V)$ or
  $\SU(V)$.  Let $X$ be a generating set for $G$, which contains a
  transvection. Then $\diam(\Cay(G,X))=\big(\log(q) n \big)^c$
  for some constant $c$ provided that
  \begin{itemize}
  \item $q\neq 9$ if $G=\Sp(V)$;
  \item $q\neq 81$ if $G=\SU(V)$;
  \item $q\neq 9$ and $q\neq 81$ if $G=\SL(V)$.
  \end{itemize}
  So Babai's bound holds for these cases.
\end{thm}
A common modification of conjectures of these types is the case of
random generators, which was also highlighted by Lubotzky
\cite{Lub}. A special case of one of the main results of a beautiful
paper by Eberhard and Jezernik (see \cite[Theorem 1.1]{EJ}) says that
if $G$ is a classical group over $\FF q$ of rank $n$ where $q$ is
bounded and $n$ is large enough then by choosing $X=\{x,y,z\}$
randomly from $G$, there is a word $w\in F_3$ of length $n^{O(1)}$
with high probability (with probability $1-e^{-cn}$ for some absolute
constant $c$) such that $w(x,y,z)$ is an element of $G$ of minimal
degree. Since the elements of minimal degree of $\SL(V), \SU(V)$ and
$\Sp(V)$ are exactly the transvections of these groups, a combined use
of this result with the main result of this paper implies.
\begin{cor}
  Let $G$ be one of $\SL(n,q),\ \Sp(n,q)$ or $\SU(n,q)$ where $q$ satisfies the
  assumptions of \cref{thm:main}. Let us also assume that $q$ is bounded.
  Let $X=\{x,y,z\}$, where $x,y,z$ are chosen randomly from $G$ (i.e.
  independently and with uniform distribution). Then
  \[
    P\big(\diam(\Cay(G,X))\leq (\log|G|)^{C}\big)\geq 1-e^{-cn}
  \]
  for some constants $c,C$.
  That is, Babai's conjecture holds for three random generators
  with high probablity when $n$ is large enough.
\end{cor}
A part of our proof can be used to show the following,
which holds even for infinite fields.
\begin{thm}\label{thm:X-contains-K-tr-sgrp}
  Let us assume that $V$ is a $n$-dimensional non-degenerate
  symplectic or hermitian vector space over the field 
  $K$, and $G=Sp(V)$ or $G=SU(V)$. Let us assume that $X$ is a generating
  set for $G$, which contains a transvection subgroup over $K_0\leq K$,
  where $K_0=K$ in the symplectic case and $|K:K_0|=2$ in the unitary case.
  Then $\diam(\Cay(G,X))\leq n^{O(1)}$ provided that $|K_0|>2$. 
\end{thm}
\section{Preliminaries}
Throughout this paper, we use the notation $\ell_X(Y)$ for the length
of $Y$ over $X$. This is defined as follows. For any $X,Y\subset G$
with $Y\subset \langle X\rangle$ let $\ell_X(Y)$ be the smallest
number $k$ such that every element of $Y$ can be written as a product
of at most $k$  elements form $X\cup X^{-1}$, that is,
\[\ell_X(Y)=\min\{k\in \NN\,|\,Y\subset (X\cup X^{-1}\cup 1)^k\}.\]

Clearly, $\ell$ has the property
$\ell_X(Z)\leq \ell_X(Y)\cdot \ell_Y(Z)$ for any $X,Y,Z\subset G$ with
$Y \subset \langle X \rangle$, $Z \subset \langle Y \rangle$, which
makes us possible to ``cut'' the proof of \cref{thm:main} into steps
providing larger and larger generating sets having stronger and
stronger properties.  Using the next lemma, as a first step we can
assume that $X$ contains only transvections.  We formulate it as a
more general statement.
\begin{lem}\label{lem:Xconj}
  In order to prove Babai's conjecture for quasisimple groups, it is
  sufficient to assume that the generating set $X$ consists of
  conjugate elements.
\end{lem}
\begin{proof}
  Let $X$ be any generating set of the quasisimple group $G$.
  Then there exists a non-central element $t \in
  X$. For $m > 0$ an integer, let
  \[Y_m := \{x_m \cdots x_1 t x_1^{-1} \cdots x_m^{-1}\ :\
  x_1,\ldots,x_m \in X \cup \{ 1 \} \}.\] Consider the ascending chain
  $\{ \langle Y_i \rangle \}_{i=1}^{\infty}$ of subgroups of
  $G$. Assume that $\langle Y_r \rangle = \langle Y_{r+1} \rangle = H$
  for some $r$ and some subgroup $H$ of $G$. It follows that $H$ is
  closed under conjugation by elements of $X$. Since $X$ is a
  generating set of $G$, the subgroup $H$ must be normal in $G$. Since
  $G$ is quasisimple, $H$ is either central or equal to $G$. The
  subgroup $H$ cannot be central, since it contains conjugates of
  $t$. We conclude that $H = G$. It also follows that the length of
  the chain $\{ \langle Y_i \rangle \}_{i=1}^{\infty}$ is at most $r
  \leq \log_2 |G|$. Thus $Y := Y_{m}$ is a generating set for $G$ for
  any integer $m$ at least $\log |G|$. Furthermore, we have
  $\ell_X(G)\leq \ell_X(Y)\cdot \ell_Y(G)\leq
  (2\log_2|G|+1)\ell_Y(G)$. So, if Babai's bound holds for
  $\diam(\Cay(G,Y))=\ell_Y(G)$, then it also holds for
  $\diam(\Cay(G,X))$ (with a slightly larger $c$).
\end{proof}
\subsection{Linear, symplectic and unitary groups}
Let $F$ be any field and $n$ a positive integer. Let $V$ be a vector
space of dimension $n$ over $F$. Usually $F$ will denote the finite
field $\mathbb{F}_{q}$ of order $q$.

The group of all linear transformations on $V$ is denoted by
$\GL(V)$ or $\GL(n,F)$ or $\GL(n,q)$ in case
$F = \mathbb{F}_{q}$. Let $f : V \times V \to F$ be a map. An element
$g$ in $\GL(V)$ is said to preserve $f$ if
$f(gu, gv) = f(u,v)$ for all $u$, $v \in V$. The set of all elements
of $\GL(V)$ preserving $f$ is called the isometry group of
$f$.

We will assume that the map $f$ is any of two types. It will be a
symplectic form, that is, a non-singular bilinear alternating form, or
it will be a unitary form, that is, a non-singular conjugate-symmetric
sesquilinear form. Note that $f$ is called an alternating form if
$f(v,v) = 0$ for all $v \in V$. From this it follows that $f$ is
skew-symmetric, that is, $f(u,v) = - f(v,u)$ for all $u$, $v \in
V$. In case $f$ is a symplectic form, $n$ must be even and the
isometry group of $f$ is the symplectic group $\Sp(V)$ or $\Sp(n,F)$
or $\Sp(n,q)$ in case $F = \mathbb{F}_{q}$. Now let $f$ be a unitary
form. Assume $q$ is the square of an integer $q_{0}$. Let $\sigma$ be
the automorphism of $F$ defined by the identity
$\sigma(\lambda) = \lambda^{q_{0}}$ for every $\lambda \in F$. The
form $f$ is called a conjugate-symmetric sesquilinear form if
$f(u, \lambda v + w) = \lambda f(u,v) + f(u,w)$ and
$f(w,u) = f(v,w)^{\sigma}$. The isometry group of a
conjugate-symmetric sesquilinear form is the general unitary group
$\GU(V)$ or $\GU(n,q)$. The group $\GU(n,q)$ has a subgroup, called
the special unitary group $\SU(V)$ or $\SU(n,q_{0})$ of index
$q_{0}+1$ consisting of all elements with determinant $1$.

Throughout the paper $G$ will denote $\Sp(V)$,
$\SU(V)$, or the special linear group $\SL(V)$. We
will assume that $n \geq 4$. In particular, $G$ is a quasisimple
group.
\subsection{Transvections}
The notion of a transvection was probably first introduced by
Artin. In this subsection we collect some basic facts about
transvections which can be found in \cite{Artin}.

In \cite[p. 160]{Artin} an element $t$ of $\GL(V)$ is called a
transvection if it keeps every vector of some hyperplane $W$ fixed and
moves any vector $v \in V$ by some vector of $W$, that is,
$t(v) - v \in W$ and $t(v)-v \ne 0$ if $v \not\in W$.

Artin determined the form of a transvection.  Let $\phi \not= 0$ be an
element of the dual space $V^{*}$ of $V$. The set of all vectors $v
\in V$ such that $\phi(v) = 0$ is a hyperplane $W$. If $\psi \in
V^{*}$ also describes $W$, then $\psi = c\cdot\phi$ for some $c \in
F^{\times}$. Let $t$ be a transvection and let $W$ be the associated
hyperplane. Let $\phi \in V^{*}$ be associated to $W$. Then $t$ has
the form $t(x) = x + \phi(x) w$ for $x \in V$ and for some fixed $w$
in $W$, that is, $\phi(w) = 0$. Conversely, any map of this form has
fixed point space the hyperplane associated to $\phi$ and every vector
in $V$ is moved by a multiple of $w$. If $t \not= 1$, then the
$1$-dimensional subspace of $W$ generated by $w$ is called the
direction of $t$.

It can be easily seen that every transvection is inside $\SL(V)$, and
they are all conjugate to each other in $GL(V)$. Moreover, if $n\geq 3$, then
they are all conjugate even in $\SL(V)$.
\subsection{Transvections in $\SL(V)$, $\Sp(V)$
  and $\SU(V)$}
Recall that throughout this paper $G$ is any of the groups $\SL(V)$,
$\Sp(V)$, $\SU(V)$ where $F$ is the finite field $\mathbb{F}_{q}$ and
$V$ is a finite dimensional vector space over $F$. In each cases, let
$\mt$ denote the set of all transvections in $G$.

In the following we borrow several notions and concepts borrowed from
\cite{HZ}.  Fix $0 \not= u \in V$ and $0 \not= \phi \in V^{*} =
\mathrm{Hom}(V,F)$ such that $\phi(u) = 0$. The element $u \otimes
\phi \in V \otimes V^{*}$ may be viewed as an endomorphism of $V$ and
$t = 1 + u \otimes \phi$ is a transvection satisfying $t(x) = x +
\phi(x) u$ for all $x \in V$. In fact, the set of all transvections in
$\SL(V)$ is
\[
\mt(\SL(V))=\{ 1 + u \otimes \phi \mid 0 \not= u \in V, 0 \not= \phi \in
V^{*}, \phi(u)=0\}.
\]
Note that $1 + u \otimes \phi = 1 + v \otimes \psi$ for
some other choice of $0 \not= v \in V$ and $0 \not= \psi \in V^{*}$
such that $\psi(v) = 0$, if and only if, $v = \lambda u$ and
$\phi =\lambda \psi$ for some nonzero $\lambda \in F$.

From now on, if we write $1 + v \otimes \psi$ for a transvection we
assume $0 \ne v \in V$ and $0 \ne \psi \in V^{*}$ with $\psi(v)=0$.

Let $f$ be a non-degenerate symplectic or conjugate-symmetric sesquilinear
form on $V$. For $u \in V$ let
$\varphi_{u} \in V^{*} = \mathrm{Hom}(V,F)$ be the map defined as
$\varphi_{u}(x) = f(u,x)$ for all $x \in V$. Let
$t = 1 + u \otimes \phi$ be a transvection and let $x$, $y \in V$
be arbitrary subject to the conditions $x \in \mathrm{ker}(\phi)$
and $y \not\in \mathrm{ker}(\phi)$. Assume that $t$ preserves
$f$. Then $f(x,y) = f(tx, ty) = f(x, y+\phi(y)u)$. From this it
follows that $\phi(y) f(x,u) = 0$, that is,
$x \in \mathrm{ker}(\varphi_{u})$. Thus
$\mathrm{ker}(\phi) = \mathrm{ker}(\varphi_{u})$ since both spaces
have dimension $n-1$. This is equivalent to saying that
$\phi = \lambda \varphi_{u}$ for some $0 \not= \lambda \in
F$. Notice also that
$f(u,u) = \varphi_{u}(u) = \lambda^{-1} \phi(u) = 0$. We conclude
that the set of transvections in $\SL(V)$ which preserve the
form $f$ is contained in the set
$\{ 1 + \lambda u \otimes \varphi_{u} \mid \lambda \in F^{\times}, 0 \not=
u \in V, f(u,u) = 0 \}$.

In case $f$ is a symplectic form, the elements of this latter set are
called symplectic transvections and are precisely the transvections
contained in $\Sp(V)$ (see \cite[Exercise 3.20]{Wilson}), so we have
\[
\mt(\Sp(V))=\{ 1 + \lambda u \otimes \varphi_{u} \mid \lambda \in
F^{\times}, 0 \not= u \in V, f(u,u) = 0 \}.
\]

Finally, in case $f$ is a conjugate-symmetric
sesquilinear form, the transvection $1 + \lambda u \otimes
\varphi_{u}$ with $\lambda \in F$, $0 \not= u \in V$ and $f(u,u) = 0$
preserves $f$ if and only if $\Tr(\lambda):=\lambda+\lambda^{q_{0}} =0$
(see \cite[Exercise 3.22]{Wilson}), so
\[
\mt(\Sp(U))=\{ 1 + \lambda u \otimes \varphi_{u} \mid  0
\not= u \in V,\,f(u,u) = 0,\;
\lambda \in F^{\times},\,\Tr(\lambda)=0\}.
\]
\subsection{Conjugating transvections with each other}

By \cref{lem:Xconj}, we can assume that $X$ contains only
transvections.  Starting from $X$, our goal is to create the conjugate
class of all the transvections. During the proof, our main tool to
achieve this goal will be to take conjugates $t_2t_1t_2^{-1}$ for some
already generated transvections $t_1$ and $t_2$.
The following lemma will be used many times during the proof.
\begin{lem} \label{lem:conj_tr} Let $t_1 := 1+a_1 \otimes \phi_1$ and $t_2
  := 1+a_2 \otimes \phi_2$ be two transvections.  Then
  \[
  t_2t_1t_2^{-1} = 1 + (a_1+\phi_2(a_1)a_2) \otimes
  (\phi_1-\phi_1(a_2) \phi_2).
  \]
\end{lem}
\begin{proof}
  For $x \in V$, we have
  \begin{align*}
  t_2t_1t_2^{-1}(x)
  & = (1+a_2 \otimes \phi_2) (1+a_1 \otimes \phi_1)
    (1-a_2 \otimes \phi_2) (x) \\
  & = (1+a_2 \otimes \phi_2) (1+a_1 \otimes \phi_1) (x-\phi_2(x)a_2) \\
  & = (1+a_2 \otimes \phi_2)
    (x-\phi_2(x)a_2+\phi_1(x)a_1-\phi_2(x)\phi_1(a_2)a_1) \\
  & = x+\phi_1(x)a_1-\phi_2(x) \phi_1(a_2) a_1 +
    \phi_2(a_1)\phi_1(x)a_2-\phi_2(a_1)\phi_2(x)\phi_1(a_2)a_2 \\
  & = x + (\phi_1(x)-\phi_1(a_2)\phi_2(x))
    \cdot (a_1+\phi_2(a_1)a_2) \\
  & = (1+(a_1+\phi_2(a_1)a_2) \otimes
    (\phi_1-\phi_1(a_2) \phi_2))(x).
  \end{align*}
  This proves the lemma.
\end{proof}
\subsection{Transvection groups}
Let $t = 1 + u \otimes \phi \in \mt$. For any
$\lambda \in \mathbb{F}_{q}^{\times}$, let the transvection
$1 + \lambda u \otimes \phi$ be denoted by $t^{\lambda}$. We have
noted that $t^{\lambda} t^{\mu} = t^{\lambda + \mu}$ for any $\lambda$
and $\mu$ in $\mathbb{F}_{q}^{\times}$. In particular,
$t^{-1} = 1 + (-u) \otimes \phi= 1 - u \otimes \phi$.

For an arbitrary subset $\Lambda \subset \mathbb{F}_{q}$, let
$t^{\Lambda} := \{t^\lambda\,|\,\lambda\in\Lambda\}$. This set is a
group if and only if $\Lambda$ is a subgroup of the additive group of
$\FF q$.

In case $G$ is a unitary group the notation $q_{0}$ was introduced to
be the prime power which is the square root of $q$. For a unified
treatment, we set $q_{0} = q$ in the cases when $G$ is a special
linear or a symplectic group. Using this notation, for any $t\in \mt$
we have that $\mathbb{F}_{q_0}$ is the largest subset $\Lambda$ of
$\mathbb{F}_q$ such that $t^\Lambda \subset \mt$. We call
$t^{\mathbb{F}_{q_0}}$ the full transvection subgroup of $G$
containing $t$. Clearly, $t^{\mathbb{F}_{q_0}}$ contains the cyclic
group $\langle t \rangle$ and the containment is proper if and only if
$q_{0}$ is not a prime. More generally, for any subfield
$K$ of $\mathbb{F}_{q_0}$, we call $t^K$ the transvection subgroup
over $K$ containing $t$.
Let $Y\subset \mt$ be any subset of transvections in $G$.  For any
subfield $K$ of $\mathbb{F}_{q_0}$, the $K$-closure of $Y$ is defined
as the set $Y^K:=\{t^\lambda\,|\,t\in Y,\,\lambda\in K^\times\}$. We say that
$Y$ is $K$-closed if $Y^K = Y$.

Let $G$ be a symplectic or a unitary group. In this case for every
singular vector $0\neq v\in V$, there is a unique associated full
transvection subgroup, which we denote by $T_v$. In the symplectic
case, let $\lambda_0=1$ and let $\mf=\mathbb{F}_{q_0}=\mathbb{F}_q$,
while in the unitary case, let $\lambda_0$ be a fixed element of
$\mathbb{F}_q$ with $\Tr(\lambda_0)=0$, and let
$\mf=\{\lambda\in\mathbb{F}_{q} \,|\,\Tr(\lambda)=0\}$. In both cases
$1+\lambda_0\cdot v\otimes \varphi_v\in \mt$. Now
$\mf=\mathbb{F}_{q_0}\lambda_0$ is a one-dimensional
$\mathbb{F}_{q_0}$-subspace of $\mathbb{F}_q$ and $T_v=(1+v\otimes
\varphi_v)^\mf= {(1+\lambda_0\cdot v\otimes
  \varphi_v)}^{\mathbb{F}_{q_0}}$.  Note that $T_v=T_{\lambda v}$ for
every singular vector $0 \ne v \in V$ and for every $\lambda \in
F_q^\times$.
\subsection{Transvection graph}
As before, let $G$ be any of the groups $\SL(V)$,
$\Sp(V)$, $\SU(V)$ and let $\mt$ be the set of all
transvections in $G$. Again, let $F$ be the finite field
$\mathbb{F}_{q}$. Let $Y$ be any subset of $\mt$. We introduce the
transvection graph $\Gamma(Y)$ (as in \cite{BH}) and its labelled
version $\lGamma(Y)$.

The directed graph $\Gamma(Y)$ has vertex set $Y$ and two vertices
$s=1+u\otimes \phi$ and $t=1+v\otimes \psi$ are connected by a
directed edge $[s,t]$ running from $s$ to $t$ if and only if
$\psi(u)\neq 0$. The set of edges in $\Gamma(Y)$ will be denoted by
$E(Y)$. We say that the edge $[s,t] \in E(Y)$ is one-way directed if
$[t,s] \notin E(Y)$, otherwise $(s,t)$ is called a two-way directed
edge. This difference of notation will become clear when we introduce
cycles.

In the symplectic and unitary case, for any two transvections
$s=1+\lambda\cdot u\otimes \varphi_u$ and $t=1+\mu\cdot v\otimes
\varphi_v$ in $Y$ we have $[s,t]\in E(Y)$ if and only if
$\varphi_v(u)=f(v,u)\neq 0$ if and only if $[t,s]\in E(Y)$. Hence
every edge in $\Gamma(Y)$ is two-way directed and $\Gamma(Y)$ can be
seen as an undirected graph. In that case $\Gamma(Y)$ is just the same
as the non-commuting graph (which was probably first mentioned in
\cite{neumann-erdos}) of $G$ restricted to the vertex set $Y$. In
general, $st\neq ts$ if and only if at least one of $[s,t]\in E(Y)$
and $[t,s]\in E(Y)$ holds. Therefore, in the linear case we can think
of $\Gamma(Y)$ as a refinement of the non-commuting graph on $Y$.

We next define the labelled transvection graph $\lGamma(Y)$.
First, for every $s\in Y$, we fix
$u_s\in V$ and $\phi_s\in V^*$ satisfying $s=1+u_s\otimes \phi_s$.
Then $\lGamma(Y)$ is a complete directed graph with label
$l(s,t)=\phi_t(u_s)\in \mathbb{F}_q$ for every $(s,t)\in Y\times Y$.
(Note that $l(s,t)$ does not only depend on $s$ and $t$ but also on how
$u_s,\phi_s,u_t,\phi_t$ were chosen.)
Note that $\Gamma(Y)$ can be derived from $\lGamma(Y)$ simply
be deleting all labels and all edges with label $0$.

In the special case when $Y=\mt$ we get the full transvection graph
$\Gamma(\mt)$ and its labelled version $\lGamma(\mt)$. For
$Y \subseteq \mt$, the graphs $\Gamma(Y)$ and $\lGamma(Y)$ can be seen
as the subgraphs of $\Gamma(\mt)$ and $\lGamma(\mt)$ induced by $Y$,
respectively.

We now define the weight of a cycle in $\lGamma(Y)$. This concept will
be the main tool in \cref{sec:closure-of-Y}.

For any integer $k \geq 2$ and for any transvections
$s_1,s_2,\ldots,s_k\in Y$, let
\[
w(s_1,s_2,\ldots,s_k):=l(s_1,s_2)l(s_2,s_3) \ldots l(s_{k-1},s_k)l(s_k,s_1)
\]
be the weight of the $k$-tuple $(s_1,\ldots,s_k)$.  If
$w(s_1,\ldots,s_k)\neq 0$ then we say that $(s_1,\ldots,s_k)$ is a
cycle (or closed path) in $\Gamma(Y)$.

Note that unlike the labels $l(s_i,s_j)$, the weight
$w(s_1,\ldots,s_k)$ depends only on the transvections
$s_1,\ldots, s_k$. In order to distinguish paths from cycles, we
use the notation $[s_1,s_2,\ldots,s_k]$ for a directed path of length $k$
and the notation $(s_1,s_2,\ldots,s_k)$ for a directed $k$-cycle.

\section{Determining groups generated by
  transvections in terms of the transvection graph}
Throughout this section let $Y\subset \mt=\mt(SL(V))$ be any fixed subset of
transvections and $H=\fl Y\fr\leq \SL(V)$.  The goal of this section is
to give answers to the following general question: How do the properties
of the (weighted) transvection graph $\Gamma(Y)$ reflect the
properties of the generated subgroup $H=\fl Y\fr \leq \SL(V)$?

In the following we give several conditions of this type.
\subsection{Determining the irreducibility of $\mathbf{H}$}
Recall that a directed graph $\Gamma$ is called strongly connected if
for every two distinct vertices $a,b$ in $\Gamma$ there exists a directed
path from $a$ to $b$. It is easy to see that strongly connectedness is
equivalent to the following condition: for every non-empty proper
subset $Z$ of vertices of $\Gamma$ there exists an edge going from a
vertex inside $Z$ to a vertex outside  $Z$.

We define the $V$-part and the $V^*$-part of $Y\subset \mt$  as
follows
\begin{align*}
  _V Y&:=\{v\in V\,|\,\exists\phi\in V^*\textrm{ s.t. }
        1+v\otimes\phi\in Y\},\\
  Y_{V^*}&:=\{\phi\in V^*\,|\,\exists v\in V\textrm{ s.t. }
             1+v\otimes\phi\in Y\}.
\end{align*}
\begin{thm}\label{thm:irred-condition}
  $H$ acts irreducibly on $V$ if and only if the following three
  conditions hold.
  \begin{enumerate}
  \item $_V Y$ generates $V$;
  \item $Y_{V^*}$ generates $V^*$;
  \item $\Gamma(Y)$ is strongly connected.
  \end{enumerate}
\end{thm}
\begin{proof}
  We may assume that the dimension of $V$ is at least $2$ and that
  $|Y| \geq 2$.

  Let $0 \neq u \in V$ and let $U := \langle h(u) : h \in H
  \rangle$. Observe that since $U$ is the smallest $H$-invariant subspace
  containing $u$, the condition
  that $H$ acts irreducibly on $V$ is equivalent to $U = V$ for every $0\neq
  u\in V$.

  Let $t_1 = 1+a_1 \otimes \phi_1$ and $t_2 = 1+a_2 \otimes \phi_2$ be
  two distinct transvections in $Y$.  Assume that there is a directed
  edge in $\Gamma(Y)$ from $t_{1}$ to $t_{2}$, that is, $\phi_2(a_1)
  \neq 0$. Observe that if $a_{1} \in U$, then $a_{2} =
  \phi_{2}(a_{1})^{-1} \cdot (t_2(a_1)-a_1) \in U$, since $U$ is
  $H$-invariant.

  Assume that the three conditions of the statement hold.

  Since $Y_{V^*}$ generates $V^*$, there exists $t = 1 + a \otimes
  \phi \in Y$ such that $\phi(u) \neq 0$. Observe that $a =
  \phi(u)^{-1} \cdot (t(u)-u) \in U$.

  Since $\Gamma(Y)$ is strongly connected, the fifth and the third
  paragraphs imply that $_V Y \subseteq U$. The first condition
  provides $V = U$.

  Assume now that $H$ acts irreducibly on $V$.

  Let $t = 1 + w \otimes \phi \in Y$ be arbitrary and let $Z$ be the
  set of all transvections in $Y$ which may be reached from $t$ by a
  directed path in $\Gamma(Y)$. Observe that $w \in {_V Z}$ and
  $\langle _V Z \rangle$ is $H$-invariant, so $\langle _V Z \rangle =
  V$. Part (1) follows. We claim that $Z = Y$. Assume for a
  contradiction that $Y \setminus Z \not= \emptyset$. Let $\psi \in (Y
  \setminus Z)_{V^*}$. Then $\psi$ vanishes on $_V Z$ and so also on
  $V$, which is impossible. Since $t$ was chosen arbitrarily,
  $\Gamma(Y)$ must be strongly connected, giving (3).

  We now prove (2). Let $L := \langle {Y_{V^*}} \rangle \leq V^*$.
  Then $Y_{V^*}$ contains a basis $\{\phi_1,\ldots,\phi_m\}$ for $L$,
  where $m \leq n= \dim(V)$. Note that
  $I := \bigcap_{j=1}^m \ker(\phi_j) = \{0\}$, for if $0 \neq v \in I$
  then $v$ is fixed by $H$ contradicting the fact that $H$ acts
  irreducibly on $V$ (and the dimension of $V$ is at least $2$). We
  have $m \le n$ subspaces of $V$ of codimension $1$
  whose intersection is trivial.  This implies that $m=n$, in other
  words $L=V^*$.
\end{proof}
\subsection{Determining the defining field for $\mathbf{H}$}
\label{sec:deffield}
The next problem we deal with is to determine the smallest subfield
$L\leq \FF q$ such that $H$ is realisable over $L$, that is, such that
$H$ is conjugate to a subgroup of $\SL(n,L)\leq \SL(V)$. It turns out
that $L$ can be determined from the weights of cycles of $\Gamma(Y)$:
\begin{prop}(\cite[Proposition 2]{Dimartino-etal}) The field generated
  by the weights of all the cycles in $\Gamma(Y)$ is equal to the
  field generated by the traces of the matrices in $H$.
\end{prop}
\begin{rem}\label{rem:weights-gen-def-field}
  If $H$ is irreducible on $V$, then the field generated by the traces
  of the matrices in $H\leq \SL(V)$ is exactly the smallest subfield
  $L$ such that $H$ is realisable over $L$ (see \cite[Corollary
  9.23]{Isaacs}). Thus, the field generated by the weights of all
  the cycles in $\Gamma(Y)$ is equal to the smallest subfield $L$ of $\FF q$
  such that $H$ conjugates into $SL(n,L)\leq SL(n,q)\simeq SL(V)$
  (see \cite[Corollary 1]{Dimartino-etal}).
\end{rem}
\begin{lem} \label{genfq} Let $G$ be one of $\SL(n,q), \Sp(n,q)$ or
  $\SU(n,q)$.  Moreover, assume that $n \geq 3$ if $G$ is the unitary
  group. Then the traces of the elements of $G$ generate
  $\mathbb{F}_q$.
\end{lem}
\begin{proof}
  Let $V$ be the underlying $n$ dimensional space over $\FF q$, so
  $G$ can be identified with $\SL(V), \Sp(V)$ or $\SU(V)$.

  First let us assume that $G=\SL(V)$ or  $\Sp(V)$.
  Let us choose a decomposition $V=U\oplus U'$ with $\dim(U)=2$.
  In case of $G=\Sp(V)$ let us also assume that $U$ is a
  non-degenerate subspace of $V$ and $U'=U^\perp$. Let $\{x,y\}$ be a basis of
  $U$, symplectic in case $G=\Sp(V)$, and let $g\in \GL(V)$ satisfying
  \[g(x)=y+(\lambda-n-2)x,\ g(y)=-x,\ g(u)=u\textrm{ for all }u\in
    {U'}.\] Then $g$ has trace equal to $\lambda$, and it belongs to
  $G$. 

  Now, let $G=\SU(V)$. Let $U$ be a $3$-dimensional
  non-degenarate subspace of $V$, so $V=U\oplus U^\perp$ and
  $U$ possesses a basis $x,y,z$ with $f(x,x)=f(y,y)=f(x,z)=f(y,z)=0,\
  f(x,y)=f(z,z)=1$. Let $b$ be a generator of the cyclic group
  $\FF q^\times$ and let $g\in \SU(V)$ be defined as
  \[g(x)=b^{-q_0} y,\quad g(y)=b x,\quad g(z)=-b^{q_0-1} z,\quad
  g(u)=u\textrm{ for all }u\in U^\perp.\] Then the trace of $g$ is
  $-b^{q_0-1}+n-3$.  Let $p^r$ be the cardinality of the subfield
  generated by the traces of elements. Note that $b^{q_0-1}\in
  \FF{p^r}$ and the order of $b^{q_0-1}$ is $q_0+1$.  Then $p^r-1$
  is divisible by $q_0+1$, which implies that $p^r=q$.
\end{proof}
\subsection{Theorems of Dickson and Wagner}
The first important result regarding subgroups
generated by transvections is due to Dickson, who gave a full
description of subgroups of $\SL(2,q)$ generated by two non-commuting
transvections.
\begin{thm} [Dickson, see {\cite[Chapter 2,
    Theorem 8.4]{Gorenstein}}] \label{thm:dicksonlemma}
  Assume $r$ is the power of an odd
  prime. Let $\delta$ be a generator of $\mathbb{F}_r$ and set
  $L = \langle \left( \begin{array}{cc} 1 & 1 \\ 0 & 1 \end{array}
  \right), \left( \begin{array}{cc} 1 & 0 \\ \delta & 1 \end{array}
  \right) \rangle$. Then we have either
  \begin{enumerate}
  \item $L=\SL(2,\mathbb{F}_r)$ or
  \item $r=9$, $|Z(L)|=2$, $L/Z(L) \cong A_5$, and $L$ contains a
    subgroup isomorphic to $SL(2,3)$.
  \end{enumerate}
\end{thm}	
Dickson also gave a full description of subgroups of $\SL(2,q)$
The following may be found in
\cite[p. 213-214]{Huppert} and \cite[p. 285]{Dickson}.

\begin{thm}[Dickson's Theorem]\label{thm:Dicksonlist}
  Let $p$ be a prime and $f$ a positive integer. The subgroups of
  $\PSL(2,p^f)$ are the following.

  \begin{enumerate}
  \item[(i)] Elementary abelian $p$-group.
  \item[(ii)] Cyclic group whose order $z$ divides $(p^{f} \pm 1)/k$,
    where $k = (p^{f}-1,2)$.
  \item[(iii)] $D_{2z}$, where $z$ is as in (ii).
  \item[(iv)] $A_{4}$, where $p > 2$ or $p = 2$ and $f \equiv 0 \pmod 2$.
  \item[(v)] $S_4$, where $p^{2f}-1 \equiv 0 \pmod {16}$.
  \item[(vi)] $A_5$, where $p = 5$ or $p^{2f}-1 \equiv 0 \pmod 5$.
  \item[(vii)] The semidirect product of an elementary abelian group of
    order $p^{m}$ and a cyclic group of order $t$, where
    $t \mid p^{m}-1$ and $t \mid p^{f}-1$.
  \item[(viii)] $\PSL(2,p^{m})$, where $m \mid f$ and
    $\PGL(2,p^{m})$, where $2m \mid f$.
  \end{enumerate}
\end{thm}
Let again $V$ be an $n$-dimensional vector space over $\FF q$. In most
of our cases $n$ will be unbounded. We will
be interested in subspaces $U$ of $V$ of bounded dimension. Let $U$ be
a subspace of $V$ of dimension $k$. Let $W$ be another subspace of $V$
such that $V = U \oplus W$. In particular, the dimension of $W$ is
$n-k$. Let $1_W$ denote the identity map on $W$. In the special case
when $W = V$ we denote $1_{W}$ by $1$.

Let $H$ be a subgroup of $\SL(U)$. Let us denote
the subgroup
\[
  \{g\in \SL(V)\,|\, g(U)=U,\,g(W)=W,\,g_U\in
  H,\,g_W=1_W\}\leq \SL(V)
\]
by $H \oplus 1_W$. Usually we use this construction for $H=SL(U)$, or
$Sp(U)$ or $SU(U)$.

In a similar way, for a matrix group
$M \leq \SL(k,q)$, let $M \oplus 1_{n-k}$ denote the subgroup
\[
  \Big\{m \oplus 1_{n-k}:=\begin{pmatrix}m&0\\0&1_{n-k}\end{pmatrix}
  \,\Big|\,m\in M\Big\},
\]
where $1_{n-k}$ denotes the identity matrix of size
$(n-k)\times(n-k)$. Let $H$ be a subgroup of $\SL(V)$ and let
$M$ be a subgroup of $\SL(k,q)$. By writing
$H \simeq M \oplus 1_{n-k}$ it is meant that there is a direct sum
decomposition $V = U\oplus W$ with $\dim(U)=k$ and a subgroup
$H_0 \simeq H$ of $\SL(U)$ such that $H = H_0 \oplus 1_W$ and the
matrix form of $H_0$ is equal to $M$ in a suitable basis of $U$.  In
the symplectic and unitary cases when we write
$H \simeq M\oplus 1_{n-k}$ we tacitly assume that the underlying
decomposition $V = U\oplus W$ is orthogonal with respect to $f$, that
is, $f_U$ is non-degenerate and $W = U^\perp$.  In these cases $H_0$
will be a subgroup of $\Sp(U)$ or $\SU(U)$.

Now, we give a partial generalisation of \cref{thm:dicksonlemma}.
\begin{thm}\label{thm:Dickson_for_subsets}
  Let $(s,t)\in \E$ be a two-way directed edge and let
  $\textrm{K}, \Lambda \subseteq \FF{q}^\times$ containing $1$. Set
  $\delta:=w(s,t)$ and $M:=\FF p(\textrm{K},\Lambda,\delta)$.  Then
  $\fl s^{\textrm{K}},t^\Lambda\fr\simeq \SL(2,M)\oplus 1_{n-2}$,
  unless $p=2$
  or
  $M=\FF 9$.
\end{thm}
\begin{proof}
  Write the transvections as $s=1+v \otimes \phi$ and
  $t=1+w \otimes \psi$.  Since $(s,t)$ is a two-way edge,
  $\phi(w) \neq 0$ and $\psi(v) \neq 0$. Therefore
  \[V = \langle v,w \rangle \oplus (\ker(\phi) \cap \ker(\psi)).\]
  Choose a basis of $V$ whose first two vectors are $\phi(w)v,w$ and
  the other vectors form a basis of the $n-2$ dimensional space
  $\ker(\phi) \cap \ker(\psi)$. Then we get
  $H=\langle s^K,t^{\Lambda} \rangle\simeq H_0\oplus 1_{n-2}$ for some
  $H_0\leq SL(2,q)$. Therefore we may assume that $n=2$,
  $V=\langle v,w \rangle$ and $H\leq SL(2,q)$.  Furthermore, in the
  basis $\{\phi(w) v,w\}$ the transvections $s^{\kappa}$,
  $t^{\lambda}$, for $\kappa \in K$, $\lambda \in \Lambda$, have
  matrix form
  \[
    \left(\begin{array}{cc} 1 & \kappa \\ 0 & 1 \end{array} \right),
    \hspace{1cm} \left(\begin{array}{cc} 1 & 0 \\ \delta \lambda &
        1 \end{array} \right)
  \]
  Now, the image of $H$
  in $PSL(2,q)$ (denoted by $\bar{H}$) is one of the groups appearing in
  \cref{thm:Dicksonlist}. $\bar{H}$ has two non-commuting $p$-subgroups, namely
  the images of $\fl s^{\textrm{K}}\fr$ and $\fl t^\Lambda\fr$.
  Thus, $\bar H$ cannot be of type (i), (ii) or (vii).
  For the remainder of the proof, let us assume that $p \neq 2$.
  Then $\bar H$ cannot be of type (iii).
  Now, let us assume that $p=3$ and $\bar H$ is any of type (iv) or (v) or (vi),
  then we must have $\textrm{K}\Lambda\subset \FF 3$.
  If $\delta\notin \FF 9$, then by \cref{thm:dicksonlemma} we have
  \[
    |H|=|\fl s,t\fr|=|SL(2,\FF 3(\delta)|\geq |SL(2,27)|>2|A_5|\geq |H|,
  \]
  a contradiction.
  Thus $\delta\in\FF 9$, which means that either $M=\FF 9$ or $M=\FF 3$.
  In the latter case we get that $H=SL(2,3)$ by \cref{thm:dicksonlemma}.
  Now let us assume that $\bar H$ is of type (vi) and $p=5$.
  Again, we must have $\textrm{K}\Lambda\subset \FF 5$.
  Similarly as before, if $\delta\notin \FF 5$, then
  $|H|=|\fl s,t\fr|=|SL(2,\FF 3(\delta)|>2|A_5|\geq |H|$, a contradiction.

  Thus, we get that $\bar{H}$ is of type (viii), so $H=SL(2,L)$ for
  some subfield $L$ of $\FF q$. Since $M$ equals to the subfield
  generated by all the weights of
  $\Gamma(\{s^{\textrm{K}},t^\Lambda\})$, we get that $L=M$ by
  \cref{sec:deffield}.
\end{proof}
The analogous question which irreducible subgroups of $SL(n,q)$ are generated by
transvections when $n\geq 3$ was solved by Wagner:
\begin{thm}[{\cite[Theorem 1.1]{Wagner}}]\label{thm:Wagner}
  Let $V$ be vector space over $\FF q$ of dimension
  $n\geq 3$ and let $H$ be a subgroup of $SL(V)$.
  Let us assume that $H$ has the following properties.
  \begin{enumerate}
  \item\label{wagneritem1} $H$ is generated by transvections,
  \item\label{wagneritem2} $H$ acts irreducibly on $V$,
  \item\label{wagneritem3} $H$ contains a transvection subgroup of
    order larger than $2$.
  \end{enumerate}
  Then $H$ is isomorphic to one of $\SL(n,L)$, $\SU(n,L)$ or $\Sp(n,L)$ for
  some subfield $L$ of $\FF q$.
\end{thm}
\begin{rem}\leavevmode
  \begin{enumerate}
  \item Since the order of any transvection is $p>2$ by our
    assumption, property \eqref{wagneritem3} will always follow
    automatically.
  \item Originally, in \cite{Wagner} this theorem was formed for
    subgroups of $\PSL(V)$ generated by elations.  (An elation is just
    the image of a transvection under the natural map
    $\SL(V)\to\PSL(V)$.) But it is easy to see that Wagner's theorem
    is translated into the above theorem when we consider subgroups of
    $\SL(V)$.
  \end{enumerate}
\end{rem}
\subsection{Determining the type of $\mathbf{H}$}
As we have seen, the weights of the cycles determine the smallest
subfield $L$ of $\FF q$ such that an irreducible subgroup of $\SL(n,q)$
generated by a set of transvections is realisable over $L$.
Now, we introduce two other parameters of cycles of $\Gamma(\mt)$.
With their help, one can determine
whether a set of transvections generate $\SL(V), \SU(V)$ or $\Sp(V)$.
\begin{defin}\label{defin:ds-du}
  For any $r_1,r_2,\ldots,r_k\in \mt$ we define
  \begin{align*}
    d_s(r_1,r_2,\ldots,r_k)&:=w(r_1,r_2,\ldots,r_k)+(-1)^{k+1}
                             w(r_k,r_{k-1},\ldots,r_1),\\
    d_u(r_1,r_2,\ldots,r_k)&:=w(r_1,r_2,\ldots,r_k)+(-1)^{k+1}
                             w(r_k,r_{k-1},\ldots,r_1)^{\sqrt q}.
  \end{align*}
  ($d_u$ is only defined if $q$ is a perfect square). We say that a cycle
  $(r_1,\ldots,r_k)\subset \mt$ is symplectic (or singular) if
  $d_s(r_1,\ldots,r_k)=0$. Similarly, we say that
  $(r_1,\ldots,r_k)\subset \mt$ is unitary if
  $d_u(r_1,r_2,\ldots,r_k)=0$.
\end{defin}
\begin{rem}\label{rem:ds-du}\leavevmode
  \begin{thmlist}
  \item In \cite{HZ}, the notation $\det_c$ was used instead of $d_s$,
    but we changed it to reflect its connection with the symplectic
    group.
  \item By the definition of $d_s$ and $d_u$, if $(s,t)$ is a
    $2$-cycle, that is, a two-way directed edge, then $(s,t)$ is always
    symplectic, while it is unitary if and only if $w(s,t)\in \FF {\sqrt q}$.
    \label[rem]{rem:ds-du-2}
  \item By an abuse of notation, we sometimes allow ourselves to say
    that an arbitrary tuple $(r_1,\ldots,r_k)$ is symplectic/unitary
    even if it is not a cycle (in neither direction).
  \item Clearly, a one-way directed cycle (i.e a cycle, which is not a
    cycle in the reverse direction) is never symplectic or unitary,
    while a $k$-tuple which is not a cycle (in both of the two
    possible directions) is both symplectic and unitary.
    \label[rem]{rem:ds-du-4}
  \end{thmlist}
\end{rem}
\begin{thm} \label{thm:structure-from-cycles}
  Let us assume that $n\geq 3$ and
  let $Z\subset \mt\subset \SL(n,q)$ be a set of transvections such that
  $Z$ generates an irreducible subgroup of $\SL(n,q)$ and the weights
  of all cycles of $\Gamma(Z)$ generate $\FF q$. Then
  $\langle Z\rangle$ is one of $\SL(n,q),\ \SU(n,q)$ or $\Sp(n,q)$ and we
  have
  \begin{align*}
    \langle Z\rangle=\Sp(n,q)&\iff\textrm{ every cycle of }\Gamma(Z)
                              \textrm{ is symplectic},\\
    \langle Z\rangle=\SU(n,q)&\iff\textrm{ every cycle of }\Gamma(Z)
                              \textrm{ is unitary},\\
    \langle Z\rangle=\SL(n,q)&\iff \Gamma(Z)\textrm{ contains both
                               a non-symplectic and a non-unitary cycle}.
  \end{align*}
\end{thm}

\begin{proof}
  By \cref{thm:Wagner} and by \cref{rem:weights-gen-def-field},
  we have that $\langle Z\rangle$ must be
  isomorphic to one of $\SL(n,q)$, $\Sp(n,q)$ or $\SU(n,q)$.

  Let us assume that $\langle Z\rangle=SU(n,q)$ and
  let $(r_1,r_2,\ldots,r_k)$ be an arbitrary cycle in $\Gamma(Z)$. Using  the
  description of unitary transvections, each $r_i$ can be written in the form
  $r_i=1+\lambda_i u_i\otimes \varphi_{u_i}$ where $u_i\in V$ is singular and
  $\Tr(\lambda_i)=\lambda_i+\lambda_i^{\sqrt q}=0$.
  Then we have
  \begin{align*}
    w(r_1,r_2,\ldots,r_k)
    &=\varphi_{u_2}(\lambda_1u_1)\varphi_{u_3}(\lambda_2u_2)\cdots
      \varphi_{u_1}(\lambda_ku_k)\\
    &=\prod_{i=1}^k\lambda_i\cdot f(u_2,u_1)f(u_3,u_2)\cdots f(u_1,u_k)\\
    &=(-1)^k\prod_{i=1}^k\lambda_i^{\sqrt q}\cdot
      f(u_1,u_2)^{\sqrt q}f(u_2,u_3)^{\sqrt q}\cdots f(u_k,u_1)^{\sqrt q}\\
    &=(-1)^k\big(f(u_1,\lambda_2u_2)f(u_2,\lambda_3u_3)\cdots
      f(u_k,\lambda_1u_1)\big)^{\sqrt q}\\
    &=(-1)^kw(r_k,r_{k-1},\ldots,r_1)^{\sqrt q},
  \end{align*}
  so $d_u(r_1,r_2,\ldots,r_k)=0$. A similar calculation shows that if $\fl Z\fr=
  Sp(n,q)$, then every cycle in $\Gamma(Z)$ is symplectic.

  Let us assume that every cycle of $\Gamma(Z)$ is symplectic. Then
  $\langle Z\rangle \leq \langle Z^{\FF q}\rangle=\Sp(n,q)$ by
  \cite[Corollary 5.3]{HZ}.

  Thus, it remains to prove that if $Z$ generates $\SL(n,q)$, then
  $\Gamma(Z)$ must contain a non-unitary cycle.  We prove this by
  using an argument which is very similar to the one given in the proof of
  \cite[Lemma 4.6]{HZ}.

  For the remainder of the proof let $q$ be a square and
  $\langle Z\rangle=\SL(n,q)$.
  Let us consider a set of tranvections $Z'\supset Z$ and let us assume that
  every cycle in $\Gamma(Z')$ is unitary.
  If there would be a one-way
  directed edge $(s,t)\in E(Z')$, then, since $\Gamma(Z')$ is strongly
  connected, there would be a directed cycle $(s,t,r_1,\ldots,r_m)$ in
  $\Gamma(Z')$. Clearly, such a cycle must be non-unitary. Therefore,
  every edge of $\Gamma(Z')$ is two-way directed. Furthermore,
  we know that $w(s,t)\in \FF {\sqrt q}$
  for every $(s,t)\in E(Z')$ by \cref{rem:ds-du-2}.
  Now, for every (two-way) directed cycle
  $(r_1,r_2,\ldots,r_k)$ in $\Gamma(Z')$ one can define
  \[
    P_u(r_1,\ldots,r_k):=
    \frac{w(r_1,r_2,\ldots,r_k)}{w(r_k,r_{k-1},\ldots,r_1)^{\sqrt q}}.
  \]
  Clearly, a directed cycle $(r_1,\ldots,r_k)$ is unitary if and only
  if $P_u(r_1,\ldots,r_k)=(-1)^k$.  One can show that under the
  assumption that every $2$-cycle is unitary, the cycle parameter
  $P_u$ is well-behaved under ``gluing'' two-way directed cycles. More
  concretely, if a directed cycle $(r_1,\ldots,r_k)$ is obtained from
  two directed cycles $(r_1,\ldots,r_i,q_1,\ldots,q_l,r_j,\ldots,r_k)$
  and $(r_i,\ldots,r_j,q_l,\ldots,q_1)$ ($i<j$) glued along their joint
  subpath $r_i,q_1,q_2,\ldots,q_l,r_j$, then
  \begin{equation}
    P_u(r_1,\ldots,r_k)=P_u(r_1,\ldots,r_i,q_1,\ldots,q_l,r_j,\ldots,r_k)\cdot
    P_u(r_i,\ldots,r_j,q_l,\ldots,q_1).\label{eq:gluing}
  \end{equation}
  (Observe that the joint subpath must be oppositely directed in the
  two cycles.)
  Indeed,
  \begin{equation*}
    \begin{split}
      &P_u(r_1,\ldots,r_i,q_1,\ldots,q_l,r_j,\ldots,r_k)\cdot
      P_u(r_i,\ldots,r_j,q_l,\ldots,q_1)\\
      & =\frac{w(r_1,\ldots,r_i,q_1,\ldots,q_l,r_j,\ldots,r_k)\cdot
        w(r_i,\ldots,r_j,q_l,\ldots,q_1)}{
        w(r_k,\ldots,r_j,q_l,\ldots,q_1,r_i,\ldots,r_1)^{\sqrt q}\cdot
        w(q_1,\ldots,q_l,r_j,\ldots,r_i)^{\sqrt q}
      }\\ &
      =\frac{w(r_1,\ldots,r_k)\cdot \prod_{s=1}^{l-1}w(q_s,q_{s+1})
        \cdot w(r_i,q_1) \cdot w(q_l,r_j)}
      {w(r_k,\ldots,r_1)^{\sqrt q}\cdot
        \prod_{s=1}^{l-1}w(q_s,q_{s+1})^{\sqrt q}
        \cdot (w(q_1,r_i) \cdot w(r_j,q_l))^{\sqrt{q}}}=
      P_u(r_1,\ldots,r_k).
    \end{split}
  \end{equation*}
  Let $(s_1,s_2)\in E(Z')$ be any (two-way) directed edge and let
  $t:=s_2s_1s_2^{-1}$. We claim that every cycle of
  $\Gamma(Z'\cup \{t\})$ is unitary. Clearly, this should be checked
  for cycles containing $t$. We prove this for $2$-cycles and for some
  $3$-cycles first.

  Let
  $s_1=1+u_1\otimes \phi_1,\ s_2=1+u_2\otimes\phi_2$, so
  $t=1+(u_1+\phi_2(u_1)u_2)\otimes
  (\phi_1-\phi_1(u_2)\phi_2)$. For any
  $r=1+v\otimes \psi\in Z'$ we have
  \begin{align*}
    w(r,t)&=(\phi_1-\phi_1(u_2)\phi_2)(v)\cdot
            \psi(u_1+\phi_2(u_1)u_2)=\phi_1(v)\psi(u_1)\\
          &+\phi_1(v)\phi_2(u_1)\psi(u_2)-
            \phi_1(u_2)\phi_2(v)\psi(u_1)
            -\phi_1(u_2)\phi_2(v)\phi_2(u_1)\psi(u_2)\\
          &=w(s_1,r)+w(r,s_1,s_2)-w(s_2,s_1,r)-w(s_1,s_2)w(r,s_2).
  \end{align*}
  Since $(s_1,r),\ (s_1,s_2),\ (r,s_2)$ are unitary,
  their weights are in $\FF{\sqrt q}$. Furthermore,
  $w(r,s_1,s_2)=-w(s_2,s_1,r)^{\sqrt q}$ implies that
  $w(r,s_1,s_2)-w(s_2,s_1,r)\in \FF {\sqrt q}$, as well. So,
  $w(r,t)\in \FF {\sqrt q}$, that is, $(r,t)$ is unitary.

  Furthermore,
  \begin{align*}
    d_u(r,t,s_2)&=(\phi_1-\phi_1(u_2)\phi_2)(v)\cdot
                  \phi_2(u_1+\phi_2(u_1)u_2)\cdot\psi(u_2)\\
                &+\Big((\phi_1-\phi_1(u_2)\phi_2)(u_2)\cdot
                  \psi(u_1+\phi_2(u_1)u_2)\cdot \phi_2(v)\Big)^{\sqrt q}\\
                &=\phi_1(v)\phi_2(u_1)\psi(u_2)-\phi_1(u_2)\phi_2(v)
                  \phi_2(u_1)\psi(u_2)\\
                &+\Big(\phi_1(u_2)\psi(u_1)\phi_2(v)\Big)^{\sqrt q}+
                  \Big(\phi_1(u_2)\phi_2(u_1)\psi(u_2)\phi_2(v)
                  \Big)^{\sqrt q}\\
                &=w(r,s_1,s_2)-w(s_1,s_2)w(s_2,r)+w(s_2,s_1,r)^{\sqrt q}
                  +w(s_1,s_2)^{\sqrt q}w(s_2,r)^{\sqrt q}\\
                &=d_u(r,s_1,s_2)-w(s_1,s_2)w(s_2,r)+w(s_1,s_2)w(s_2,r)=0.
  \end{align*}
  A similar calculation shows that $d_u(r,t,s_1)=0$.

  Finally, let $(r_1,r_2,\ldots,r_k,t)$ be any cycle in
  $\Gamma(Z'\cup \{t\})$.  Then both $r_1$ and $r_k$ are connected
  with at least one of $s_1$ and $s_2$. Indeed, writing
  $r_i=1+v_i \otimes \psi_i$ for every $i$, the label of $(r_k,t)$ is
  $\phi_1(v_k)-\phi_1(u_2)\phi_2(v_k)$ and it is nonzero, so at least
  one of $\phi_1(v_k)$ and $\phi_2(v_k)$ must be nonzero. Depending on
  the role of $s_1$ and $s_2$, there are two possibilities:
  \begin{center}
    \tikznewcycleunitary
  \end{center}
  Using the gluing property,
  \[
    P_u(r_1,r_2,\ldots,r_k,t)=\frac{P_u(r_1,r_2,\ldots,r_k,s_1)}{
      P_u(s_1,r_1,t)\cdot P_u(s_1,t,r_k)
    }=\frac{(-1)^{k+1}}{(-1)^3\cdot (-1)^3}=(-1)^{k+1}
  \]
  in the first case and
  \begin{align*}
    P_u(r_1,r_2,\ldots,r_k,t)&=\frac{P_u(r_1,r_2,\ldots,r_k,s_2,s_1)}{
      P_u(s_1,r_1,t)\cdot P_u(s_1,t,s_2)\cdot P_u(s_2,t,r_k)
    }\\ &=\frac{(-1)^{k+2}}{(-1)^3\cdot (-1)^3\cdot (-1)^3}=(-1)^{k+1}
  \end{align*}
  in the second case. Hence $(r_1,r_2,\ldots,r_k,t)$ is unitary.

  Starting from $Z$ we can construct new transvections by a repeated
  conjugation of previously constructed transvections with each other.
  By our previous argument, if every cycle of $\Gamma(Z)$ would be
  unitary, then we could never get a non-unitary cycle with such
  repeated conjugations. However, since $\langle Z\rangle=\SL(n,q)$ and
  all the transvections are conjugate in $\SL(n,q)$, we finally get all
  the transvections. Clearly, the full transvection graph
  $\Gamma(\mt)$ contains also non-unitary cycles, which proves that
  $\Gamma(Z)$ must contain a non-unitary cycle.
\end{proof}
\section{The proof of the main theorem}
By \cref{lem:Xconj} and the previous section, in order to prove \cref{thm:main},
we can assume that the following properties hold for $X$.
\begin{enumerate}
\myitem[(P1)] $G$ is any of $\SL(V)$, $\Sp(V)$,
  $\SU(V)$ and $X\subset \mt(G)$.\label{con-G}
\myitem[(P2)] The sets $_V X$ and $X_{V^*}$ generate
  $V$ and $V^*$ respectively and $\Gamma(X)$ is strongly connected.
  \label{con-sc}
\myitem[(P3)] The weights of all the cycles of
  $\Gamma(X)$ generate $\FF q$. \label{con-genfield}
\myitem[(P4)] If $G=Sp(V)$, then every cycle in
  $\Gamma(\mt)$ is symplectic. \label{con-symp}
\myitem[(P5)] If $G=SL(V)$, then $\Gamma(X)$
  contains a non-unitary cycle. \label{con-nonunit}
\end{enumerate}
Note that all of these properties remain true when we add new transvections to
$X$.
\subsection{Decreasing the diameter of $\Gamma(X)$}
\label{sec:dec-diameter}
The goal of this section is to obtain a set of transvections
such that the diameter of the associated transvection graph is small.
\begin{lem} \label{lem:smalldiam} There is a set of transvections
  $X'\supset X$, with $\ell_X(X') \leq O(n)$ and with the
  following property.  For any $s,t\in\mt$ with $[s,t]\notin\E$
  there is an $r\in X'$ such
  that $[s,r],[r,t]\in \E$.  In other words,
  $\diam(\Gamma(X''))\leq 2$ for every
  $X' \subseteq X'' \subseteq \mt$.
\end{lem}
\begin{proof}
  Let $s,t \in \mathcal{T}$ be two transvections of the form
  $s=1+v \otimes \phi$, $t=1+w \otimes \psi$ with $[s,t]\notin\E$. If
  there is no edge going from $s$ to any element of $X$ then
  $\nu(v)=0$ for every $\nu \in X_{V^*}$, which contradicts the fact
  that $\langle X_{V^*} \rangle = V^*$ since $v \neq 0$. Similarly, if
  there is no edge going from any element of $X$ to $t$ then
  $\psi(u)=0$ for every $u \in {_VX}$, which contradicts the fact that
  $\langle {_VX} \rangle = V$ since $\psi \neq 0$. Since $\Gamma(X)$
  is strongly connected, there exists a path $[s,r_1,\ldots,r_k,t]$ in
  $\Gamma(X)$, where $r_1,\ldots,r_k \in X$. Choose such path of
  minimal length.

  We claim that $k \leq n+1$. To prove this, write
  $r_i = 1+v_i \otimes \phi_i$ for every $i=1,\ldots,k$. Observe that, by
  minimality of $k$, $\phi_{i+1}(v_j)=0$ unless $j=i$ or (possibly)
  $j >i+1$, and $\phi_{i+1}(v_i) \neq 0$ for every
  $i=1,\ldots,k-1$. Let $\alpha_1, \ldots, \alpha_{k-1} \in F$ be such
  that $\sum_{j=1}^{k-1} \alpha_j v_j = 0$. For any
  $i \in \{1,\ldots,k-1\}$ we have
  \[
    0 = \phi_{i+1} \left( \sum_{j=1}^{k-1} \alpha_j v_j \right)
    = \sum_{j=i}^{k-1} \alpha_j \cdot \phi_{i+1}(v_j). \hspace{1cm}
    (\mbox{Eq.\ i})
  \]
  Starting from equation $k-1$ and going backwards, using that
  $\phi_{i+1}(v_i)\neq 0$ for all $i=k-1,k-2,\ldots,1$, it is clear that
  $\alpha_j=0$ for all $j=1,\ldots,k-1$. This implies that
  $\{v_1,\ldots,v_{k-1}\}$ is linearly independent, therefore
  $k \leq n+1$.

  We claim that $s,r,t$ is a path in $\Gamma(\mathcal{T})$,
  where
  \[r=r_kr_{k-1} \cdots r_2 r_1 r_2^{-1} \cdots
    r_{k-1}^{-1}r_k^{-1}.\]
  Observe that this claim will conclude the proof because $k \leq n+1$, so
  adding the element $r$ to $X$ for every $s,t\in \mt$ we get a set of
  transvections $X'$ with $\ell_X(X')\leq 2n+1$ satisfying the required
  property.
	
  Define
  $t_i := r_i r_{i-1} \cdots r_2 r_1 r_2^{-1} \cdots r_{i-1}^{-1}
  r_i^{-1}$ for all $i$ with $1 \leq i \leq k$, in particular
  $t_1=r_1$ and $t_k=r$. Write $t_i = 1+w_i \otimes \psi_i$. We claim that
  $\psi_i(v) \neq 0$ for every $i \geq 1$. If $i=1$ then this is
  clear.  Assume now $i \geq 2$. By \cref{lem:conj_tr},
  \[
    t_i = r_i t_{i-1} r_i^{-1} = 1+(w_{i-1}+\phi_i(w_{i-1})v_i)
    \otimes (\psi_{i-1}-\psi_{i-1}(v_i) \phi_i).
  \]
  We deduce that there exist nonzero scalars $\lambda_i$ such that
  $w_1=\lambda_1 v_1$, $\psi_1 = \lambda_1^{-1} \phi_1$ and, if
  $i \geq 2$, then
  \[
    \begin{array}{l} w_i = \lambda_i (w_{i-1}+\phi_i(w_{i-1})v_i) \in
      \langle v_1,\ldots,v_i \rangle, \\
      \psi_i = \lambda_i^{-1} (\psi_{i-1}-\psi_{i-1}(v_i) \phi_i),
    \end{array}
  \]
  therefore $\psi_i(v) = \lambda_i^{-1} \psi_{i-1}(v) \neq 0$ by using induction
  on $i$.

  We also need $\psi(w_k) \neq 0$. Since
  $w_{k-1} \in \langle v_1,\ldots,v_{k-1} \rangle$, we have
  \[
    \psi(w_k) = \psi(\lambda_k (w_{k-1}+\phi_k(w_{k-1})v_k)) =
    \lambda_k \phi_k(w_{k-1}) \psi(v_k).
  \]
  We will prove by induction that $\phi_{i+1}(w_i) \neq 0$ for
  every $i=1,\ldots,k-1$. Note that $w_1=\lambda_1 v_1$ hence
  $\phi_2(w_1) = \lambda_1 \phi_2(v_1) \neq 0$. Now assume $i \geq
  2$. Then, since $w_{i-1} \in \langle v_1,\ldots,v_{i-1} \rangle$,
  assuming $\phi_i(w_{i-1}) \neq 0$, we have
  \[
    \phi_{i+1}(w_i) =
    \phi_{i+1}(\lambda_i(w_{i-1}+\phi_{i}(w_{i-1})v_{i})) =
    \lambda_i \phi_{i}(w_{i-1}) \phi_{i+1}(v_{i}) \neq 0.
  \]
  This concludes the proof.
\end{proof}
Since $X'$ has small length over $X$, we may replace $X$ by the set
$X'$ and keep calling it $X$. By the previous lemma, from now on
we can assume the property
\begin{enumerate}
\myitem[(P6)] The (directed) diameter of
  $\Gamma(X')$ is at most $2$ for any $X'\supseteq X$. \label{con-diamXatmost2}
\end{enumerate}
In what follows, we say that a directed path $[r_1,r_2,\ldots,r_k]$ in
$\Gamma(\mt)$ is two-way directed, if it is also a directed path in
the reverse direction, i.e. if $(r_i,r_{i+1})$ is two-way directed for
every $1\leq i<k$. In a similar way, we can also define the concept of
two-way directed cycles, as well.
A transvection graph $\Gamma(Y)$ is said to be two-way connected,
if for every vertices $r,s\in Y$, there is a two-way directed path
in $\Gamma(Y)$ connecting $r$ and $s$. If $\Gamma(Y)$ is two-way connected,
then the two-way diameter of $\Gamma(Y)$
is defined as the smallest $k$ such that every two
vertices in $Y$ are connected in $\Gamma(Y)$
by a two-way directed path of length at most $k$.

Note that if $G$ is either $\SU(V)$ or $\Sp(V)$, then the diameter of
$\Gamma(X)$ coincides with its two-way diameter because in these cases
every edge is a two-way directed edge. So the next claim is only
interesting for $G=SL(V)$.
\begin{lem}\label{thm:double_edge_dist}
  There exists a set of transvections $X'$, containing $X$, such that
  $\ell_X(X') \leq O(1)$ with the following property.  For every
  $s,t\in \mt$ there are $r_1,\ldots,r_k \in X'$ with $k\leq 5$ such
  that $s,r_1,\ldots,r_k,t$ is a two-way directed path. In other
  words, the two-way directed diameter of $\Gamma(X'')$ is at most $6$
  for every $X' \subseteq X'' \subseteq \mt$.
\end{lem}
\begin{proof}
  First, we increase $X$ to have the property that there is a two-way
  directed edge from any $s\in \mt$ into $X$.  Let us assume that
  $s\in \mt$ is an arbitrary vertex which is not connected to any
  element of $X$ by a two-way directed edge. Since $\langle X_{V^*}
  \rangle = V^*$, there is an $r\in X$ with $[s,r]\in \E$. By property
  \ref{con-diamXatmost2}, the diameter of $X \cup \{s\}$ is at most
  $2$, so there is a $t\in X$ such that $[r,t,s]$ is a directed
  path. Let us write $s=1+u\otimes\phi,\ r=1+v\otimes\psi$ and
  $t=1+w\otimes \chi$.  Let
  $t_s=trt^{-1}=1+(v+\chi(v)w)\otimes(\psi-\psi(w)\chi)$.  By our
  assumption $\phi(v)=\chi(u)=0$, while
  $w(s,r,t)=\psi(u)\chi(v)\phi(w)\neq 0$. Thus, we have
  \[
    \phi(v+\chi(v)w)=\phi(w)\chi(v)\neq 0
    \textrm{ and }(\psi-\psi(w)\chi)(u)=\psi(u)\neq 0,
  \]
  so $(s,t_s)$ is a two-way directed edge. Adding $t_s$ to $X$ for
  every $s$ which is not connected to $X$ by a two-way directed edge,
  the resulting $X_1$ will have the required property.

  In view of the previous condition, we would like to increase $X_1$ to
  an $X'$ such that the two-way distance of any $r_1,r_2\in X_1$ in
  $\Gamma(X')$ is at most $4$. (Note that this property only implies
  that the two-way diameter of $X'$ is at most $6$.)  By property
  \ref{con-diamXatmost2}, for any $r_1,r_2\in X_1$ there is an $r\in X_1$ such
  that $[r_1,r,r_2]$ is a directed path.  Therefore, it is enough to prove
  that we can increase $X_1$ to $X'\subset \mt$ (whose length over $X$
  is small enough) such that for any one-way directed edge
  $[r_1,r_2]\in X_1$ there is a $t\in X'$ such that $[r_1,t,r_2]$ is a
  two-way directed path. After this the proof will become complete.

  Let $r_1=1+u_1\otimes\phi_1$ and $r_2=1+u_2\otimes \phi_2$ be
  elements of $X_1$ such that $[r_1,r_2]$ is a one-way directed edge,
  i.e.\ $\phi_2(u_1)\neq 0,\ \phi_1(u_2)=0$.
  By property \ref{con-diamXatmost2},
  there is a $t=1+v\otimes\psi\in X$ such that $[r_2,t,r_1]$ is a
  directed path of length $2$. If this is a two-way path, then there
  is nothing to be done. Otherwise, we distinguish two cases:

  \textbf{Case 1:} Exactly one of $[r_2,t]$ and $[t,r_1]$ is one-way directed.
  
  Let us assume that, say, $[r_2,t]$ is a one-way directed edge but
  $(t,r_1)$ is a $2$-cycle.  Let
  $t_1:=r_1tr_1^{-1}=1+(v+\phi_1(v)u_1)\otimes(\psi-\psi(u_1)\phi_1)$.
  Now,
  \begin{align*}
    &(\psi-\psi(u_1)\phi_1)(u_1)= \psi(u_1)\neq 0,&&
      \phi_1(v+\phi_1(v)u_1)=\phi_1(v)\neq 0,   \\
    &(\psi-\psi(u_1)\phi_1)(u_2)= \psi(u_2)\neq 0,&&
      \phi_2(v+\phi_1(v)u_1)=\phi_1(v)\phi_2(u_1)\neq 0,
  \end{align*}
  so $[r_2,t_1,r_1]$ is two-way directed.

  \textbf{Case 2:} Both of the edges in the path
  $[r_2,t,r_1]$ are one-way directed.

  Let again $t_1=r_1tr_1^{-1}$.
  The above calculation now shows that $[r_2,t_1,r_1]$ is a directed
  path such that $(r_2,t_1)$ is a $2$-cycle, while $[t_1,r_1]$ is a
  one-way directed edge. Now, the same argument as in Case 1 can be
  applied but to the element $t_2=r_2t_1r_2^{-1}$ to get a two-way directed
  path $[r_2,t_2,r_1]$.
\end{proof}
Again, we may replace $X$ by $X'$ to ensure the following hereditary property.
\begin{enumerate}
  \myitem[(P7)] The two-way diameter of $\Gamma(X')$ is at most $6$
  for any $X'\supseteq X$.
  \label{con-2waydiamXatmost6}
\end{enumerate}
\subsection{Generating the $\FF{q_0}$-closure of $X$}
\label{sec:closure-of-Y}
Throughout this section let $X\subset G$ be a set of transvections possessing
all properties from \ref{con-G} to \ref{con-2waydiamXatmost6}.
The goal of this section is to generate the $\FF{q_0}$-closure of $X$ in
short length over $X$.

Our main tool here is the concept of weight of cycles.  In what
follows, when we talk about a directed cycle $(r_1,r_2,\ldots, r_k)$
in $\Gamma(\mt)$, we generally mean that the indices of its vertices are
elements of $\ZZ_k$, so $r_{k+1}=r_1,\;r_{k+2}=r_2$, etc.

For every integer $k \geq 2$, let $L_k=L_k(X)$ be the subfield of
$\mathbb{F}_q$ generated by the weights of the cycles in $\Gamma(X)$
of length at most $k$. The sequence $\{L_i\}_{i=2}^{\infty}$ is an
increasing sequence of subfields in $\mathbb{F}_q$.
\begin{lem}\label{lem:short-cycles-gen-Fq}
  Assume that the integer $k \geq 3$ is such that $L_{k-1} < L_k$. Then
  \begin{enumerate}
  \item $k \in \{3,4,5\}$.
  \item If $k \in \{4,5\}$ and $(r_1,\ldots,r_k)$ is a cycle whose
    weight is not in $L_{k-1}$, then there is an index $i$ with
    $1\leq i\leq k$ such that $[r_i,r_{i+2}]\notin E(X)$.
  \end{enumerate}
\end{lem}
\begin{proof}
  Let $(r_1,\ldots,r_k)$ be a $k$-cycle in $\Gamma(X)$ whose weight is
  not in $L_{k-1}$. Choosing indices $i$ and $j$ with
  $1\leq i <j\leq k$, we claim that at least one of the following
  holds:
  \begin{itemize}
  \item $[r_i,r_{i+1},\ldots, r_j]$ is a path of minimum length in $\Gamma(X)$
    from $r_i$ to $r_j$;
  \item $[r_j,r_{j+1},\ldots, r_k,r_1,\ldots, r_i]$ is a path of minimum
    length in $\Gamma(X)$ from $r_j$ to $r_i$.
  \end{itemize}
	
  Assume that the claim is false. Then $j \geq i+2$ and
  $(i,j) \not= (1,k)$. Let $x_1,\ldots,x_l$, $y_1,\ldots,y_m\in X$
  such that $[r_i,x_1,\ldots,x_l,r_j]$ is a shorter path than
  $[r_i,r_{i+1},\ldots, r_j]$, that is, $l<j-i-1$, and
  $[r_j,y_1,\ldots,y_m, r_i]$ is a shorter path than the path
  $[r_j,r_{j+1},\ldots, r_k,r_1,\ldots, r_i]$, that is, $m<k-j+i-1$.
  So we have the following picture:
  \begin{center}
    \tikzhatszog
  \end{center}
  Let us consider the four directed cycles of the above picture.
  By assumption, the three cycles
  \begin{align*}
    C_1&:=(r_i,x_1,\ldots,x_l,r_j,r_{j+1},\ldots,r_{i-1}),\\
    C_2&:=(r_i,r_{i+1},\ldots,r_{j-1},r_j,y_1,\ldots,y_m),\\
    C_3&:=(r_i,x_1,\ldots,x_l,r_j,y_1,\ldots,y_m)
  \end{align*}
  all have lengths smaller than $k$, so their
  weights are inside $L_{k-1}$. However,
  \[
    w(r_1,\ldots,r_k)=\frac{w(C_1)\cdot w(C_2)}{w(C_3)},
  \]
  so $w(r_1,\ldots,r_k)\in L_{k-1}$, a contradiction. This completes
  the proof of the claim.

  Now, let us assume that $k\geq 6$ and let us choose $i=1,\ j=4$. By
  the previous paragraph, either $[r_1,r_2,r_3,r_4]$ is a path of
  shortest length between $r_1$ and $r_4$ or
  $[r_4,r_5,r_6,\ldots,r_1]$ is a path of shortest length between
  $r_4$ and $r_1$, in $\Gamma(X)$.  However, both of these paths have
  length at least $3$, while the diameter of $\Gamma(X)$ is at most
  $2$ by \cref{lem:smalldiam}, a contradiction. So $k\leq 5$, as
  claimed.

  Now, let us assume that $k\in \{4,5\}$. If, say,
  $[r_1,r_3]\in E(X)$, then $[r_1,r_2,r_3]$ is not the shortest path
  from $r_1$ to $r_3$, so, using the first paragraph of the proof
  again, we get that $[r_3,\ldots,r_k,r_1]$ must be a shortest path
  from $r_3$ to $r_1$ in $\Gamma(X)$. Since $\diam(\Gamma(X))\leq 2$
  by \cref{lem:smalldiam}, we get that $k=4$ and
  $[r_3,r_1]\notin E(X)$. Hence the last claim follows with $i=3$.
\end{proof}
In what follows for any subset $Y\subset \mt$ we use the notation
$Y^{(k)}$ for the set of all transvections which can be written as a
product of at most $k$ many elements of $Y\cup Y^{-1}$, that is,
$Y^{(k)}:=\{t\in\mt\,|\,\ell_Y(t)\leq k\}$.

Using the previous lemma along with property
\ref{con-genfield}, we get that $L_5(X)=\FF q$.  The following lemma
allows us to assume that $L_3(X)=\FF q$ and $|\FF q:L_2(X)| \leq 2$.
\begin{lem}\label{lem:weights-twoway-edges}
  There exists a set of transvections $X'$ with $\ell_X(X')=O(1)$ such
  that $L_3(X')=\FF q$ and $|\FF q:L_2(X')| \leq 2$.
\end{lem}
\begin{proof}
  Let $k \geq 3$ be an integer. Let
  $r_1=1+u_1\otimes \phi_1,r_2=1+u_2\otimes\phi_2,
  \ldots,r_k=1+u_k\otimes\phi_k\in\Gamma(X)$ be arbitrary
  transvections.  Besides that let $a_i:=\phi_{i+1}(u_i)$ for every
  index $i$ with $1\leq i\leq k$ and let
  $b:=\phi_k(u_{k-2}),\ c:=\phi_1(u_{k-1})$ and
  $d:=\phi_{k-1}(u_k)$. Note that these are certain labels in the
  transvection graph.  We have the following labelled graph:
  \begin{center}
    \tikzshortencycle
  \end{center}
  (Note that only those edges appear on this picture,
  which have roles in the forthcoming arguments.)
  By \cref{lem:conj_tr},	
  \begin{align*}
    r_kr_{k-1}r_k^{-1}
    &=1+(u_{k-1}+\phi_k(u_{k-1})u_k)\otimes(\phi_{k-1}-\phi_{k-1}(u_k)\phi_k)\\
    &=1+(u_{k-1}+a_{k-1}u_k)\otimes(\phi_{k-1}-d\phi_k).
  \end{align*}
  Now, we calculate the weight of $(r_1,\ldots,r_{k-2},r_kr_{k-1}r_k^{-1})$.

  \begin{align*}
    w(r_1,\ldots,r_{k-2},r_kr_{k-1}r_k^{-1})
    &=\prod_{i=1}^{k-3}a_i\cdot\Big((\phi_{k-1}-d\phi_k)(u_{k-2})\Big)
      \Big(\phi_1(u_{k-1}+a_{k-1}u_k)\Big)\\
    &=\prod_{i=1}^{k-3}a_i\cdot(a_{k-2}-db)(c+a_{k-1}a_k)
  \end{align*}
  Assume now that $(r_1,\ldots,r_k)$ is a cycle of minimal length with
  the property that its weight is not in $L_{k-1}(X)$. Observe that
  $w(r_1,\ldots,r_{k-1}) \in L_{k-1}(X)$ because, if it is nonzero,
  then it is the weight of a $(k-1)$-cycle. We would like to apply the
  above process to $(r_1,\ldots,r_k)$. By
  \cref{lem:short-cycles-gen-Fq}, we have $k\leq 5$. Let us assume
  that $k>3$. Then by using \cref{lem:short-cycles-gen-Fq} again, we
  can also assume that $b=\phi_k(u_{k-2})=0$.  Then we have
  \begin{align*}
    w(r_1,\ldots,r_{k-2},r_kr_{k-1}r_k^{-1})
    &=\prod_{i=1}^{k-3}a_i\cdot a_{k-2}(c+a_{k-1}a_k)\\
    &=w(r_1,\ldots,r_k)+w(r_1,\ldots,r_{k-1}).
  \end{align*}
  By our assumption, $w(r_1,\ldots,r_k)\not\in L_{k-1}(X)$ but
  $w(r_1,\ldots,r_{k-1})\in L_{k-1}(X)$, so
  \[w(r_1,\ldots,r_{k-2},r_kr_{k-1}r_k^{-1}) \not \in L_{k-1}(X).\]
  Since $r_kr_{k-1}r_k^{-1}\in X^{(3)}$ we get that
  the weights of the cycles in $\Gamma(X^{(3)})$
  of length at most $4$ generate $\FF q$. Using this argument once again,
  we get that $L_3(X^{(9)}) = \FF q$.

  It remains to prove that we can extend $X^{(9)}$ to an $X'$ with
  $\ell_X(X')=O(1)$ such that $|\FF q:L_2(X')|\leq 2$.  If
  $L_2(X^{(27)})=\FF q$, then we can choose $X'=X^{(27)}$, so let us
  assume that $L_2(X^{(27)}) \neq \FF q$ for the remainder.

  Since $L_3(X^{(9)})=\FF q$, we know that the weights of the $3$
  cycles in $\Gamma(X^{(9)})$ along with $L_2(X^{(9)})$ generate $\FF q$.  Let
  $(r_1,r_2,r_3)$ be any $3$-cycle in $\Gamma(X^{(9)})$, whose weight
  is not inside $L_2(X^{(9)})$. Using the same notation as in the
  first part of the proof, we have the following diagram.
  \begin{center}
    \tikzthreecycle
  \end{center}
  Now, $r_1,r_2,r_3,r_3r_2r_3^{-1}\in X^{(27)}$, so the weights
  \begin{align*}
    m_1:=w(r_1,r_2)&=a_1c,&& m_2:=w(r_2,r_3)=a_2d\\
    m_3:=w(r_3,r_1)&=a_3b,&& m_4:=w(r_1,r_3r_2r_3^{-1})=(a_1-db)(c+a_{2}a_3)
  \end{align*}
  are elements of $L_2(X^{(27)})$.
  Let $\delta=a_1a_2a_3=w(r_1,r_2,r_3)\notin L_2(X^{(9)})$. Then we have
  \[
    m_4=(a_1-db)(c+a_2a_3)=m_1+\delta-\frac{m_1m_2m_3}{\delta}-m_2m_3,
  \]
  so $\delta$ is a root of the polynomial
  $x^2+(m_1-m_4-m_2m_3)x-m_1m_2m_3\in L_2(X^{(27)})[x]$. Thus, we get
  that the weight of any $3$-cycle in $\Gamma(X^{(9)})$ is in a second
  degree extension of $L_2(X^{(27)})$.  But such weights along with
  $L_2(X^{(27)})$ generate $\FF q$, and $\FF q$ contains only at most
  one second degree extension of $L_2(X^{(27)})$. This readily implies
  that $|\FF q:L_2(X^{(27)})| = 2$.
\end{proof}
Now, we show that if $G\neq \SL$, then
$L_2(X)=\FF {q_0}$ can be achieved very easily. (Recall that $q_0=q$
unless $G=SU(n,q)$, when $q_0=\sqrt q$.)  In contrast, proving this in
case of $G=\SL$ seems to be much more difficult (see
\cref{sec:gluing}).

\begin{lem}\label{lem:SympUnit_M=Fq0}
  If $L_3(X)=\FF q$ and that $|\FF q:L_2(X)| \leq 2$, then
  $L_2(X^{(3)}) = \FF {q}$ in the symplectic case and $L_2(X) = \FF
  {q_0}$ in the unitary case.
\end{lem}
\begin{proof}
  By our assumption, we already know that $|L_2(X)|\geq \sqrt q$.
  Now, in the unitary case the claim follows, since $L_2(\mt)\leq \FF {q_0}$
  holds by a combined use of \cref{thm:structure-from-cycles} and
  \cref{rem:ds-du-2}.

  In the symplectic case, we take the last part of the previous proof.
  Let $r_i=1+\lambda_iu_i\otimes \varphi_{u_i}\in X$ for $i=1,2,3$.
  In this case we have
  \begin{align*}
    dbc
    &=\varphi_{u_2}(\lambda_3 u_3)\varphi_{u_3}(\lambda_1 u_1)
      \varphi_{u_1}(\lambda_2 u_2)\\
    &=\lambda_1\lambda_2\lambda_3(-\varphi_{u_3}(u_2))(-\varphi_{u_1}(u_3))
      (-\varphi_{u_2}(u_1))=-a_1a_2a_3
  \end{align*}
  hence
  $(r_1,r_3r_2r_3^{-1})$ is a $2$-cycle in  $X^{(3)}$
  whose weight is
  \begin{align*}
    w(r_1,r_3r_2r_3^{-1})
    &=(a_1-db)(c+a_2a_3)=a_1a_2a_3-dbc+a_1c-dba_2a_3\\
    &=2w(r_1,r_2,r_3)+w(r_1,r_2)-w(r_1r_3)w(r_2r_3).
  \end{align*}
  Since the characteristic of $\FF q$ is different from $2$, we have that
  $w(r_1,r_2,r_3)$ is an element of the subfield generated by the $2$-cycles
  of $\Gamma(X^{(3)})$ for any $3$-cycle $(r_1,r_2,r_3)$ of $\Gamma(X)$.
  Thus, the weights of the $2$-cycles of $\Gamma(X^{(3)})$ generate $\FF q$.
\end{proof}
Note that the above proof already used that $q$ is odd, when $G$ is
a symplectic group. Our next proof hevily relies on
\cref{thm:Dickson_for_subsets}, so it uses our full
assumption on $q$.
\begin{lem}\label{lem:M-closure-of-X}
  Let $L:=L_2(X)\leq \FF {q_0}$ and let us assume that $L\neq \FF 9$.
  Then we have the following.
  \begin{thmlist}
  \item $\ell_X(X^L)\leq O((\log|L|)^c)$;\label[lem]{lem:M-closure-of-X-1}
  \item If $X$ contains a transvection subgroup  $s_0^{L_0}$
    over $L_0$ for some subfield $L_0\leq L$, then
    $\ell_X(X^L)\leq O(|L:L_0|^c)$.\label[lem]{lem:M-closure-of-X-2}
  \end{thmlist}
\end{lem}
\begin{proof}
  First, we show that if we could generate a transvection subgroup
  $s^L$ for some $s\in X$, then we can generate $X^L$ in length
  $\ell_{X\cup s^L}(X^L)=O(1)$.  Indeed, let us assume that such an
  $s^L$ is already generated and let $(s,r)$ be a two-way directed
  edge in $\Gamma(X)$.  Using \cref{thm:Dickson_for_subsets}, we know
  that $\fl s^L,r\fr\simeq \SL(2,L)\oplus 1_{n-2}$.  In
  particular, $r^{L}\leq \fl s^L,r\fr$.  By using
  \cref{cor:Babai-strong}, we get the following for some constant $c_0$.
  \begin{align*}
    \ell_{X\cup s^L}(r^{L})
    \leq \ell_{\{s^L,r\}}(s^{L})\leq
      \ell_{\{s^L,r\}}(\SL(2,L))
    =O\Big(\frac{\log|\SL(2,L)|}{\log|L|}\Big)^{c_0}=O(1).
  \end{align*}
  Now, for an arbitrary $t\in X$ let $s,r_1,r_2,\ldots r_k,t$ be a
  two-way directed path in $\Gamma(X)$ with $k\leq 5$. (such a path
  exists by property \ref{con-2waydiamXatmost6}.) Using the above
  argument repeatedly to the two-way edges
  $(s,r_1),\;(r_1,r_2),\;\ldots,(r_k,t)$, we get that
  $\ell_{X\cup s^L}(t^L)=O(1)$, as claimed.

  Now, we turn to the problem of generating a transvection subgroup
  over $L$.  If $|L|= 3$, the statement is trivial, so for the
  remainder we can assume that $|L| \neq 3, 9$. Then there is a
  $2$-cycle $(s_1,t_1)$ in $\Gamma(X)$ such that
  $\FF p(w(s_1,t_1))\nleq \FF 9$. Let $K_1:=\FF p(w(s_1,t_1))\leq L$.
  Using \cref{thm:Dickson_for_subsets} along with
  \cref{cor:Babai-strong}, we get that
  \begin{align*}
    \ell_{X}(s_1^{K_1})
    &\leq \ell_{\{s_1,t_1\}}(s_1^{K_1})\leq \ell_{\{s_1,t_1\}}(SL(2,K_1))\\
    &=O((\log_2|SL(2,K_1)|)^{c_0})=O((\log_2|K_1|)^{c_0})
  \end{align*}
  for some constant $c_0$. Now, if $K_1=L$, then we are done.

  Now assume that $K_1 < L$ and let $(s_2,t_2)$ be a two-way directed
  edge in $\Gamma(X)$ with $w(s_2,t_2)\notin K_1$ such that the
  distance of $s_1$ and $s_2$ is the smallest possible, and let
  $K_2:=K_1(w(s_2,t_2))$.  Using the same procedure as in the first
  paragraph to a shortest two-way path $s_1,r_1,\ldots,r_k,s_2$, we
  can generate $s_2^{K_1}$ in length
  $\ell_{X \cup s_1^{K_1}}(s_2^{K_1})=O(1)$.  (Note that
  $w(s_1,r_1),\;w(r_1,r_2),\;\ldots,w(r_k,s_2)$ are all elements of
  $K_1$ by our assumption.) Then
  $\fl s_2^{K_1},t_2\fr=\SL(2,K_2)\oplus 1_{n-2}\geq s_2^{K_2}$ by
  \cref{thm:Dickson_for_subsets} so, by using
  \cref{cor:Babai-strong}), we get that
  \[
    \ell_{X \cup s_1^{K_1}}(s_2^{K_2})\leq
    O(1)\cdot\ell_{\{s_2^{K_1},t_2\}}(s_2^{K_2}) \leq
    O\Big(\frac{\log|\SL(2,K_2)|}{\log|K_1|}\Big)^{c_0}
    =O(|K_2:K_1|^{c_0}).
  \]
  In general, if $K_i< L_2(X)$ and the $K_i$-closure of some $s_i\in X$ is
  already generated, then we choose $(s_{i+1},t_{i+1})$ in $\Gamma(X)$
  with $K_{i+1}:=K_i(w(s_{i+1},t_{i+1}))>K_i$ and we apply the above
  procedure to get $s_{i+1}^{K_{i+1}}$ in length
  $\ell_{X \cup s_i^{K_i}}(s_{i+1}^{K_{i+1}})=
  O(|K_{i+1}:K_i|^{c_0})$.  Finally, we get a strictly increasing
  chain of subfields $K_1<K_2<\ldots<K_m=L$ with
  $m\leq \log\log|L|$ and a $2$-cycle $(s_i,t_i)$ in $\Gamma(X)$ with
  $K_{i+1}=K_{i}(w(s_{i+1},t_{i+1}))$ for every $i<m$.
  Using the above procedure we can
  generate $s_m^L$ in length
  \begin{align*}
    \ell_X(s_m^{L})
    & = O((\log|K_1|)^{c_0})\prod_{i=1}^{m-1}O(|K_{i+1}:K_i|^{c_0}) \\
    & = O(\log|L|^{c_0})O(1)^{\log\log|L|}\leq O((\log|L|)^{c}).
  \end{align*}
  So the proof of the first claim is complete.

  Finally, the second claim follows by using essentially the same argument, but
  we do not use the second paragraph. Instead, we
  choose $s_1:=s_0$ and $K_1:=L_0$.
\end{proof}
Now, we are in the position to generate the $\FF {q_0}$-closure of
$X$ in short length over $X$ in many cases.
\begin{cor}\label{cor:M-closure-summary}
  Let $M=\FF{q_0}$ unless $G=SL(n,q)$ with $q$ perfect square, when
  let $M=\FF{\sqrt q}$. Then  $\ell_X(X^M)=O(\log(q)^c)$.
  Replacing $X$ with $X^M$ we can assume that
  \begin{enumerate}
    \myitem[(P8)]\label{con-M-closed-X}
    \begin{enumerate}
    \item If $G=\Sp(n,q)$ or $G=\SU(n,q)$ or $q$ is not a perfect square,
      then $X$ is $\FF{q_0}$-closed.
    \item If $G=\SL(n,q)$ and $q$ is a perfect square, then $X$ is $M$-closed
      with $M=\FF{\sqrt q}$ and $L_3(X)=\FF q$.
    \end{enumerate}
  \end{enumerate}
\end{cor}
\begin{proof}
  By \cref{lem:weights-twoway-edges,lem:SympUnit_M=Fq0}, we can assume
  that $L_2(X)=\FF{q_0}$ unless $G=SL(n,q)$ with $q$ perfect square
  when we can assume that $L_2(X) \geq \FF{\sqrt q}$ and
  $L_3(X)=\FF q$.  Now, application of \cref{lem:M-closure-of-X-1}
  gives the result.
\end{proof}

\subsection{Gluing triangles}\label{sec:gluing}
The ultimate goal of this section is to finish the proof of
\cref{thm:main} for the case when $G=\SL(V)$.
In view of the previous sections, we
can assume that $X$ satisfies properties \ref{con-G}--\ref{con-2waydiamXatmost6}
and \ref{con-M-closed-X}/(b). Therefore,
$X$ is $M$-closed for $M:=\FF{\sqrt q}$ and
$X$ contains a triangle whose weight is not in $M$.
By our assumptions on $q$, we have $|M|\geq 5$.

The main difficulty we face now is to construct a
$2$-cycle $(s,t)$ in short length over $X$ such that $w(s,t)\notin M$.
In fact, we will be able to generate such a $2$-cycle in length
$\ell_X(s,t)=O(1)$. To prove this, we will use the concept of non-unitary
cycles (see \cref{defin:ds-du}).

We highlight that in the foregoing discussion, a large part of our
calculations remain valid only until the weight of each $2$-cycle of
the examined parts of the transvection graph are inside $M$, that is,
until each such $2$-cycle is unitary (see \cref{rem:ds-du-2}). Since
our primary goal is to construct a non-unitary $2$-cycle, it does not
cause any problem, if we tacitly assume that during our proof all
intermediate two-way directed edges are unitary $2$-cycles.

By \cref{thm:structure-from-cycles}, $\Gamma(X)$ must contain a
non-unitary cycle. Previously (see
\cref{lem:short-cycles-gen-Fq,lem:weights-twoway-edges}) we have seen
how to construct a triangle of bounded length over $X$ whose weight is
outside $M$. Using a similar argument, we can also construct a
non-unitary triangle of bounded length over $X$.
\begin{lem}
  We can assume that $X$ contains a non-unitary triangle.
\end{lem}
\begin{proof}
  First, if $X$ contains a one-way directed edge $[r,s]\in E(X)$, then
  by property \ref{con-diamXatmost2}
  there is a $t\in X$ such that $(r,s,t)$ is a one-way
  directed cycle, which is non-unitary by
  \cref{rem:ds-du-4}. Therefore, for the remainder of this proof we
  can assume that every edge in $\Gamma(X)$ is two-way directed.

  Let $(r_1,r_2,\ldots,r_k)$ be a non-unitary cycle of minimal length
  in $\Gamma(X)$. (Such a cycle exists by property \ref{con-nonunit}.)
  By our assumption, every $2$-cycle is unitary in $\Gamma(X)$, so we
  can use the gluing property of $P_u$ (see \cref{eq:gluing}), to
  conclude that $(r_1,r_2,\ldots,r_k)$ cannot be obtained by gluing
  two shorter (and, therefore, unitary) cycles of $\Gamma(X)$.

  Therefore, we get that $(r_1,r_2,\ldots,r_k)$ must be  a chordless
  cycle of length $k\leq 5$, where ``chordless'' means that
  $[r_i,r_j]$ is an edge if only if $i-j\equiv \pm 1 \pmod k$.
  If $k>3$, then let $s:=r_kr_{k-1}r_k^{-1}$. Now,
  $(r_1,r_2,\ldots,r_{k-2},s)$ is a (two-way)
  directed cycle of length $k-1$. Furthermore, $s$ is
  connected by both $r_k$ and $r_{k-1}$ with a two-way edge.
  Using the gluing property we have
  \begin{align*}P_u(r_1,\ldots,r_k)
    &=P_u(r_1,r_2,\ldots,r_{k-2},s)\cdot P_u(r_1,s,r_k)\\
    &\cdot P_u(r_k,s,r_ {k-1})\cdot P_u(r_{k-1},s,r_{k-2 }).
  \end{align*}

  Thus, we get a non-unitary cycle of length at most $k-1\leq 4$.  If
  we still do not have a non-unitary triangle, we use the
  above argument again for a non-unitary $4$-cycle. So, we conclude that
  $X^{(9)}$ surely contains  a non-unitary triangle.
\end{proof}
For the remainder of this section,
if $s_1,s_2,s_3$, and so on, are transvections, we use the usual notation
$s_i=1+u_i \otimes \phi_i$.
Recall that, by \cref{lem:conj_tr}
\[
  s_i^{s_j} = s_j^{-1}s_is_j = 1+(u_i-\phi_j(u_i)u_j)
  \otimes (\phi_i+\phi_i(u_j)\phi_j).
\]

\begin{lem}\label{lem:conjiso}
    For any $g\in G$, the conjugation by $g$ defines a graph isomorphism on
    $\Gamma(\mt)$. Moreover,
    $w(r_1,r_2,\ldots,r_k)=w(r_1^g,r_2^g,\ldots,r_k^g)$ for any
    $r_1,r_2,\ldots,r_k\in \mt$.
  \end{lem}
  \begin{proof}
    Both claims can be easily proved using the fact that if
    $r=1+u\otimes\phi$ is a transvection, then
    $r^g = g^{-1} r g = 1+g^{-1}(u)\otimes(\phi\circ g)$.
\end{proof}
\begin{lem} \label{lem:red1}
  Let $s_1,s_2,s_3,s_4,s_5$ be transvections such that the following happens.
  
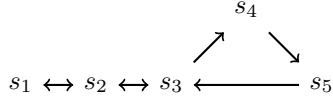
\[\begin{tikzpicture}
    \draw[<->,thick] (0.3,0) -- (0.7,0);
    \draw[<->,thick] (1.3,0) -- (1.7,0);
    \draw[->,thick] (2.3,0.3) -- (2.7,0.7);
    \draw[->,thick] (3.3,0.7) -- (3.7,0.3);
    \draw[->,thick] (3.7,0) -- (2.3,0);
    \draw (0,0) node {$s_1$};
    \draw (1,0) node {$s_2$};
    \draw (2,0) node {$s_3$};
    \draw (3,1) node {$s_4$};
    \draw (4,0) node {$s_5$};
  \end{tikzpicture}\]
  Assume that $[s_1,s_3]$ is not a double edge. If
  $[s_1,s_3]$ is a single edge, then the triangle $(s_1,s_3,s_2)$ is
  non-unitary. If $[s_3,s_1]$ is a single edge then the triangle
  $(s_1,s_2,s_3)$ is non-unitary. Assume now that $[s_1,s_3]$ and
  $[s_3,s_1]$ are not edges. Then the following happens.
  
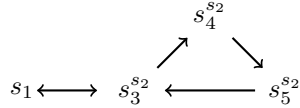
\[\begin{tikzpicture}
    \draw[<->,thick] (-0.3,0) -- (0.5,0);
    \draw[->,thick] (1.3,0.3) -- (1.7,0.7);
    \draw[->,thick] (2.3,0.7) -- (2.7,0.3);
    \draw[->,thick] (2.6,0) -- (1.4,0);
    \draw (-0.5,0) node {$s_1$};
    \draw (1,0) node {$s_3^{s_2}$};
    \draw (2,1) node {$s_4^{s_2}$};
    \draw (3,0) node {$s_5^{s_2}$};
  \end{tikzpicture}\]
\end{lem}
\begin{proof}
  If one of $[s_1,s_3]$ and $[s_3,s_1]$ is a single edge then the
  result follows from \cref{rem:ds-du-4}. We have
  $s_3^{s_2} = 1+(u_3-\phi_2(u_3)u_2) \otimes
  (\phi_3+\phi_3(u_2)\phi_2)$, hence
  \begin{align*}
    (\phi_3+\phi_3(u_2) \phi_2)(u_1) =
                          \phi_3(u_2) \phi_2(u_1) \neq 0,\\
    \phi_1(u_3-\phi_2(u_3) u_2) =
                          -\phi_2(u_3) \phi_1(u_2) \neq 0.
  \end{align*}
  Therefore $(s_1,s_3^{s_2})$ is a double edge.
  By \cref{lem:conjiso}, we are done.
\end{proof}
\begin{lem} \label{lem:abc-acb}
  If $s_1,s_2,s_3$ are transvections then
  \[w(s_1,s_2,s_3)-w(s_1,s_3,s_2) =
    w(s_1,s_3)-w(s_1,s_2)w(s_2,s_3)-w(s_1^{s_2},s_3).\]
\end{lem}
\begin{proof}
  Note that
  $s_1^{s_2} = 1+(u_1-\phi_2(u_1)u_2) \otimes (\phi_1+\phi_1(u_2) \phi_2)$, so
  \begin{align*}
    w(s_1^{s_2},s_3)
    & = (\phi_3(u_1)-\phi_2(u_1)\phi_3(u_2))
      \cdot (\phi_1(u_3)+\phi_1(u_2)\phi_2(u_3)) \\
    & = w(s_1,s_3)+w(s_1,s_3,s_2)-w(s_1,s_2,s_3)-w(s_1,s_2)w(s_2,s_3).
  \end{align*}
  The result follows.
\end{proof}
\begin{rem}\label{rem:abc-acb}
  \cref{lem:abc-acb} implies that whenever we have a triangle
  $(s_1,s_2,s_3)$ such that $w(s_1,s_2,s_3)\notin M$ but
  $w(s_1,s_3,s_2)\in M$, then there are $s,t\in\mt$ with
  $w(s,t)\notin M$ and $\ell_X(s,t)\leq 3$, which is our main goal.
  So, for the rest of this section we can assume that whenever
  $w(s_1,s_2,s_3)\notin M$ for some already generated
  $s_1,s_2,s_3\in\mt$, then $w(s_1,s_3,s_2)$ is also not in $M$.
\end{rem}
The following lemma will be used many times in this section without referring
to it.

\begin{lem} \label{lem:af}
  Let $s_1,s_2,s_3,s_4 \in Y$, $\lambda \in M$ and let
  \begin{align*}
    A & := w(s_1,s_2,s_3), \\
    B & :=w(s_1,s_2,s_4,s_3)-w(s_1,s_2,s_3,s_4), \\
    C & :=-w(s_3,s_4) \cdot w(s_1,s_2,s_4), \\
    D & :=d_u(s_1,s_2,s_3), \\
    E & :=w(s_1,s_2,s_4,s_3)-w(s_1,s_2,s_3,s_4)+
        w(s_2,s_1,s_4,s_3)^{\sqrt{q}}-w(s_2,s_1,s_3,s_4)^{\sqrt{q}}, \\
    F & :=-w(s_3,s_4) \cdot d_u(s_1,s_2,s_4).
  \end{align*}
  Then
  \[w(s_1,s_2,s_3^{s_4^{\lambda}})=A+\lambda B+\lambda^2 C,\]
  \[d_u(s_1,s_2,s_3^{s_4^{\lambda}}) = D+ \lambda E + \lambda^2 F\]
  and we have the following.
  \begin{enumerate}
  \item If either $A \in M$, $C \not \in M$ or $A \not \in M$,
    $C \in M$ then for every $\lambda \in M^{\times}$ at least one
    of $w(s_1,s_2,s_3^{s_4^{\lambda}})$,
    $w(s_1,s_2,s_3^{s_4^{-\lambda}})$ does not belong to $M$.
  \item If at least one of $D$, $E$, $F$ is nonzero then there are at
    most two values of $\lambda$ for which
    $d_u(s_1,s_2,s_3^{s_4^{\lambda}})=0$.
  \end{enumerate}
\end{lem}
\begin{proof}
  Observe that
  $s_3^{s_4^{\lambda}}=1+(u_3-\lambda \phi_4(u_3)u_4) \otimes
  (\phi_3+\lambda \phi_3(u_4) \phi_4)$. Therefore
  \begin{align*}
    w(s_1,s_2,s_3^{s_4^{\lambda}})
    & = \phi_2(u_1) \cdot (\phi_3(u_2)+\lambda \phi_3(u_4)\phi_4(u_2)) \cdot
      (\phi_1(u_3)-\lambda \phi_4(u_3) \phi_1(u_4)) \\
    & = \phi_2(u_1) \phi_3(u_2) \phi_1(u_3)-
      \lambda \phi_2(u_1)\phi_3(u_2)\phi_4(u_3)\phi_1(u_4) \\
    & +\lambda \phi_2(u_1)\phi_3(u_4)\phi_4(u_2)\phi_1(u_3)-
      \lambda^2 \phi_2(u_1)\phi_3(u_4)\phi_4(u_2)\phi_4(u_3)\phi_1(u_4)\\
    & = A+\lambda B+\lambda^2 C.
  \end{align*}
  In particular
  $w(s_1,s_2,s_3^{s_4^{\lambda}})+w(s_1,s_2,s_3^{s_4^{-\lambda}}) =
  2A+2\lambda^2 C$.
  Using the permutation $(12)$ to the indices in the formulas given for $A,B,C$,
  an easy calculation shows that
  \[d_u(s_1,s_2,s_3^{s_4^\lambda})=w(s_1,s_2,s_3^{s_4^{\lambda}})+
    w(s_2,s_1,s_3^{s_4^{\lambda}})^{\sqrt q}=D+\lambda E+\lambda^2F.\]
  The remaining claims of the lemma are clear.
\end{proof}
\begin{lem} \label{lem:red2} Assume we are in the case $G=\SL(V)$ and
  that $|M| = \sqrt{q}$. Assume at least one of the following holds.
  \begin{thmlist}
  \item There exist $s_1,s_2,s_3 \in X$ with 
    $w(s_1,s_2,s_3) \not \in M$ and $d_u(s_1,s_2,s_3) \neq 0$.
    \label[lem]{lem:red2-1}
  \item There exist $s_1,s_2,s_3,s_4 \in X$ with 
    $w(s_1,s_2,s_3) \not \in M$ and $d_u(s_4,s_3,s_2) \neq 0$.
    \label[lem]{lem:red2-2}
  \end{thmlist}
  Then there exist $t_1,t_2 \in \mathcal{T}$ such that
  $w(t_1,t_2) \not \in M$ and $\ell_X(\{t_1,t_2\})=O(1)$.
\end{lem}
\begin{proof}
  Assume case (1) holds. We have a triangle $(s_1,s_2,s_3)$ with the
  property that $w(s_1,s_2,s_3) \not \in M$ and
  $d_u(s_1,s_2,s_3) \neq 0$. Note that we can assume that
  $w(s_1,s_2,s_3)-w(s_1,s_3,s_2) \in M$ by \cref{lem:abc-acb}, so
  $w(s_1,s_2,s_3)+w(s_1,s_3,s_2) \neq 0$. Indeed, if
  $w(s_1,s_2,s_3)+w(s_1,s_3,s_2)=0$ then
  \[
    2 w(s_1,s_2,s_3)=w(s_1,s_2,s_3)-w(s_1,s_3,s_2)\in M
  \]
  and being the characteristic odd, this would imply
  $w(s_1,s_2,s_3) \in M$, a contradiction.

  Define $W:=\langle u_1,u_2,u_3 \rangle_{\mathbb{F}_q}$. The fact
  that $w(s_1,s_2,s_3)+w(s_1,s_3,s_2) \neq 0$ implies that
  $\dim_{\mathbb{F}_q}(W)=3$ (in other words, $u_1$, $u_2$ and $u_3$
  are linearly independent over $\mathbb{F}_q$) and that we have a
  direct sum decomposition
  \[
    V = \langle u_1,u_2,u_3 \rangle \oplus (\ker(\phi_1) \cap
    \ker(\phi_2) \cap \ker(\phi_3)).
  \]
  This is because if a linear combination
  $\mu_1 u_1 + \mu_2 u_2 + \mu_3 u_3$ (where
  $\mu_1,\mu_2,\mu_3 \in \mathbb{F}_q$) belongs to $\ker(\phi_1)\cap
  \ker(\phi_2)\cap\ker(\phi_3)$ (in particular, this happens if
  this linear combination is zero) then applying $\phi_1$, $\phi_2$
  and $\phi_3$ to such a linear combination we obtain a homogeneous
  linear system in $\mu_1,\mu_2,\mu_3$ whose determinant is precisely
  $w(s_1,s_2,s_3)+w(s_1,s_3,s_2)$. Moreover, a very similar argument
  shows that $\phi_1$, $\phi_2$ and $\phi_3$ are linearly independent,
  and therefore generate the dual space $W^*$.

  Since $(s_1,s_2,s_3)$ is a $3$-cycle, the transvection graph
  $\Gamma(\{s_1,s_2,s_3\})$ is strongly connected, therefore the group
  $H=\langle s_1^M,s_2^M,s_3^M \rangle \leq GL(V)$ is identified (via
  the direct sum decomposition above) with an irreducible subgroup of
  $GL(W)$ (by \cref{thm:irred-condition}). By \cref{thm:Wagner} and by
  \cref{sec:deffield}, $H$ is a special linear, unitary or symplectic
  group of dimension $3$ over $M(w(s_1,s_2,s_3)=\FF q$. But since the
  $3$-cycle $(s_1,s_2,s_3)$ is non-symplectic and non-unitary, the
  group $H$ cannot be symplectic nor unitary, so $H \cong \SL(3,q)$.

  Let $t_1,t_2\in \mt(H)\subset\mt$ with $w(t_1,t_2)\notin M$.
  By \cref{cor:Babai-strong},
  \[
    \ell_X(t_1,t_2)\leq \ell_{\{s_1^M,s_2^M,s_3^M\}}(H)
    =O\Big(\frac{\log|SL(3,q)|}{\log(3|M|)}\Big)^c=O(1),
  \]
  as claimed.

  \medskip

  Assume case (2) holds. Now, we can assume that $w(s_1,s_3,s_2)\notin M$
  by \cref{rem:abc-acb}. By applying a permutation  to the indeces if necessary,
  we can assume that $(s_1,s_2,s_3)$ and
  $(s_2,s_4,s_3)$ are two triangles, with $w(s_1,s_2,s_3)\notin M$ and
  $d_u(s_2,s_4,s_3)\neq 0$. We can also assume that $d_u(s_1,s_2,s_3)=0$,
  furthermore $w(s_2,s_3,s_4) \in M$ and $w(s_2,s_4,s_3) \in M^\times$
  by case (1). We call a triangle ``good'' if it is non-unitary and its weight
  is not in $M$. If we find a good triangle we are done by case (1),
  so what we will do is to look for a good triangle.  By
  \cref{rem:ds-du-4}, $(s_1,s_2)$, $(s_2,s_3)$ and
  $(s_3,s_1)$ are double edges. So we have the following picture.
  
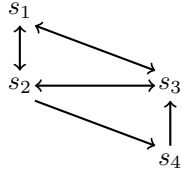
\[
    \begin{tikzpicture}
      \draw[<->,thick] (0.2,-0.2) -- (1.8,-0.8);
      \draw[<->,thick] (0,-0.2) -- (0,-0.8);
      \draw[<->,thick] (0.2,-1) -- (1.8,-1);
      \draw[<-,thick] (2,-1.2) -- (2,-1.8);
      \draw[<-,thick] (1.8,-1.8) -- (0.2,-1.2);
      \draw (0,0) node {$s_1$};
      \draw (0,-1) node {$s_2$};
      \draw (2,-1) node {$s_3$};
      \draw (2,-2) node {$s_4$};
    \end{tikzpicture}
  \]
  In each of the following cases we use \cref{lem:af} for a well-chosen
  triangle.

  \medskip

  \noindent\textbf{Case 1:} $[s_4,s_1]\notin\E$ or $w(s_2,s_4,s_1)\notin M$.\\
  We have $w(s_2,s_4,s_1^{s_3^\lambda})= A+B\lambda+C\lambda^2$ with
  $A=w(s_2,s_4,s_1),\ B=w(s_2,s_4,s_3,s_1)-w(s_2,s_4,s_1,s_3)$ and
  $C= -w(s_1,s_3) \cdot w(s_2,s_4,s_3) \in M^\times$.

  If $[s_4,s_1]\notin\E$, then $A=0$ and
  \[
    B=w(s_4,s_3,s_1,s_2) =
    \frac{w(s_1,s_2) w(s_3,s_1) w(s_2,s_4,s_3)}{w(s_1,s_3,s_2)} \notin M,
  \]
  so $w(s_2,s_4,s_1^{s_3^\lambda})\notin M$ for every
  $\lambda\in M^\times$.

  On the other hand, if $A \notin M$, then since $C \in M$,
  $w(s_3,s_1) \cdot d_u(s_2,s_4,s_3)\neq 0$ and $|M| \geq 5$,
	there exists $\lambda\in M$
  such that $(s_2,s_4,s_1^{s_3^\lambda})$ is a good triangle.

  \medskip

  \noindent\textbf{Case 2:} $[s_4,s_1]\in\E$ and
  ($[s_1,s_4]\in \E$ or $[s_3,s_4]\notin \E$). \\
  We have $w(s_3,s_2,s_1^{s_4^{\lambda}})=A+B\lambda +C\lambda^2$
  with $A = w(s_3,s_2,s_1)\notin M$ and
  $C=-w(s_1,s_4) \cdot w(s_3,s_2,s_4)\in M$. In particular, at
  least one of $w(s_3,s_2,s_1^{s_4^{\lambda}})$ and
  $w(s_3,s_2,s_1^{s_4^{-\lambda}})$ is not in $M$ for any $\lambda\in M$.

  Now, we have
  $d_u(s_3,s_2,s_1^{s_4^{\lambda}}) = \lambda E - \lambda^2 F$
  for some $E,F\in\FF q$.
  If $[s_1,s_4]\in\E$, then $F=w(s_1,s_4) \cdot d_u(s_2,s_4,s_3)\neq 0$,
  so there are at most two values $\lambda\in M$ such that
  $d_u(s_3,s_2,s_1^{s_4^{\lambda}})=0$. On the other hand, if
  $[s_1,s_4] \notin\E$ and $[s_3,s_4]\notin\E$, then
  $d_u(s_3,s_2,s_1^{s_4^{\lambda}}) = \lambda E$, where
  \begin{align*}
    E
    &=w(s_3,s_2,s_4,s_1)-w(s_3,s_2,s_1,s_4) +
      w(s_2,s_3,s_4,s_1)^{\sqrt{q}}-w(s_2,s_3,s_1,s_4)^{\sqrt{q}}\\
    &=w(s_3,s_2,s_4,s_1)\neq 0,
  \end{align*}
  so $d_u(s_3,s_2,s_1^{s_4^{\lambda}})\neq 0$ for every
  $\lambda\in M^\times$.  Since $|M|\geq 5$, there is a
  $\lambda\in M$ such that $(s_3,s_2,s_1^{s_4^{\lambda}})$ is a good
  triangle.

  \medskip

  \noindent\textbf{Case 3:} $[s_4,s_1]$ is a one-way directed edge,
  $(s_4,s_3)$ is a two-way directed edge, $w(s_1,s_2,s_4)\in M^\times$.\\
  Since $A=w(s_1,s_2,s_3) \not \in M$, $C=-w(s_3,s_4)\cdot
  w(s_1,s_2,s_4) \in M$,
  $F=-w(s_3,s_4) \cdot d_u(s_1,s_2,s_4)\neq 0$ and $|M| \geq 5$, we can find
  $\lambda \in M$ such that the triangle
  $(s_1,s_2,s_3^{s_4^{\lambda}})$ is good.
\end{proof}

From now on, if a triangle is unitary and its weight
is not in $M$ then we call it of ``type 1'' and if a triangle is
non-unitary and its weight is in $M$ then we call it of ``type
2''. Note that the edges of a triangle of type 1 are two-way directed
by \cref{rem:ds-du-4}.

Assuming that there are triangles of both types, what we want is to
find one of the two situations described in \cref{lem:red2} in the
transvection graph spanned by a power of $X$ whose exponent is bounded
by a constant. So in the following discussion we may assume
(by a way of contradiction) that in every power
of $X$ whose exponent is bounded by a constant, all non-unitary
triangles are of type 2 and all triangles whose weight is not in $M$
are of type 1. Under this assumption, our goal is to generate four
transvections satisfying \cref{lem:red2-2} whose length over $X$ is bounded.

Assume we have two triangles $T_1=(s_1,s_2,s_3)$ and
$T_2=(s_4,s_5,s_6)$, with $T_1$ of type 1 and $T_2$ of type 2. By
\cref{rem:ds-du-4}, $T_1$ consists of double edges.
First, we want to reduce the problem to the case in which $T_1$
and $T_2$ share a vertex.

Let us assume that $T_1$ and $T_2$ do not have a vertex in common.
Let, say, $s_3$ and $s_4$ be chosen such that their two-way distance
in $\Gamma(X)$ is the smallest possible one, and let
$s_3,r_1,\ldots,r_k,s_4$ be a shortest two-way path in $\Gamma(X)$, so
$k\leq 5$ by property \ref{con-2waydiamXatmost6}. If $k\geq 1$, then we
may apply \cref{lem:red1} to change $T_2$ to a triangle of type $2$,
whose distance from $T_1$ is shorter. Using this process at most 5
times, we may assume that we are in the following situation.
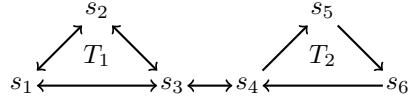
\[\begin{tikzpicture}
\draw[<->,thick] (0.2,0.2) -- (0.8,0.8);
\draw[<->,thick] (1.2,0.8) -- (1.8,0.2);
\draw[<->,thick] (1.8,0) -- (0.2,0);
\draw[<->,thick] (2.2,0) -- (2.8,0);
\draw[->,thick] (3.2,0.2) -- (3.8,0.8);
\draw[->,thick] (4.2,0.8) -- (4.8,0.2);
\draw[->,thick] (4.8,0) -- (3.2,0);
\draw (0,0) node {$s_1$};
\draw (1,1) node {$s_2$};
\draw (1,0.4) node {$T_1$};
\draw (4,0.4) node {$T_2$};
\draw (2,0) node {$s_3$};
\draw (3,0) node {$s_4$};
\draw (4,1) node {$s_5$};
\draw (5,0) node {$s_6$};
\end{tikzpicture}\]
If $d_u(s_5,s_6,s_3)\neq 0$, then we may replace $T_2$ to
$(s_5,s_6,s_3)$, while if $w(s_1,s_2,s_4) \notin M$
then we may replace $T_1$ with $(s_1,s_2,s_4)$ and reduce to the case
in which $T_1$ and $T_2$ share a vertex.
Now, let us assume that $d_u(s_5,s_6,s_3)=0$ and $w(s_1,s_2,s_4) \in M$.
Then $d_u(s_5,s_6,s_3^{s_4^{\lambda}})=\lambda E+\lambda^2F$ with
$F=w(s_3,s_4) \cdot d_u(s_5,s_6,s_4) \neq 0$, so
there is at most one $\lambda\in M^\times$ such that
$d_u(s_5,s_6,s_3^{s_4^{\lambda}})=0$. Since $ |M|\geq 5$, we can choose
$\lambda\in M^\times$ such that both $d_u(s_5,s_6,s_3^{s_4^{\lambda}})\neq 0$
and $d_u(s_5,s_6,s_3^{s_4^{-\lambda}})\neq 0$.

Since $w(s_1,s_2,s_3) \notin M$ but $w(s_1,s_2,s_4) \in M$, at least one of
$w(s_1,s_2,s_3^{s_4^{\lambda}})$, $w(s_1,s_2,s_3^{s_4^{-\lambda}})$ is
not in $M$. Up to replacing $\lambda$ with $-\lambda$, we may assume
that $w(s_1,s_2,s_3^{s_4^{\lambda}}) \notin M$. By our choice of $\lambda$,
$d_u(s_3^{s_4^{\lambda}},s_5,s_6) \neq 0$, so at least one of
$(s_3^{s_4^{\lambda}},s_5,s_6)$, $(s_3^{s_4^{\lambda}},s_6,s_5)$ is a
triangle of type 2.

So we are in the situation where $T_1$ and $T_2$ share a vertex as it
can be seen on Figure \ref{fig:twotriangles}.
\begin{figure}[h]
\begin{tikzpicture}
\draw[<->,thick] (0.2,0.2) -- (0.8,0.8);
\draw[<->,thick] (1.2,0.8) -- (1.8,0.2);
\draw[<->,thick] (1.7,0) -- (0.3,0);
\draw[->,thick] (2.2,0.2) -- (2.8,0.8);
\draw[->,thick] (3.2,0.8) -- (3.8,0.2);
\draw[->,thick] (3.8,0) -- (2.3,0);
\draw (0,0) node {$s_1$};
\draw (1,1) node {$s_2$};
\draw (1,0.4) node {$T_1$};
\draw (3,0.4) node {$T_2$};
\draw (2,0) node {$s_3$};
\draw (3,1) node {$s_4$};
\draw (4,0) node {$s_5$};
\end{tikzpicture}
\caption{ Triangles of type 1 and 2. \newline
Type 1: unitary, weight not in $M$. \newline
Type 2: non-unitary, weight in $M$.}
\label{fig:twotriangles}
\end{figure}

We consider several cases, which are distinguished by the directed edges
appearing on these two triangles.

\medskip

\noindent\textbf{Case 1:} $[s_2,s_5]$, $[s_3,s_5]$, $[s_2,s_4]$, $[s_5,s_1]$,
$[s_4,s_1]$, $[s_4,s_3]$ are not edges in Figure \ref{fig:twotriangles}.

We would like to prove
that $w(s_1^{s_3},s_2^{s_3},s_5^{s_4}) \not \in M$ and that
$d_u(s_1^{s_3},s_5^{s_4},s_3^{s_4})\neq 0$.

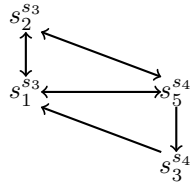
\[\begin{tikzpicture}
\draw[<->,thick] (0.2,-0.2) -- (1.8,-0.8);
\draw[<->,thick] (0,-0.2) -- (0,-0.8);
\draw[<->,thick] (0.2,-1) -- (1.8,-1);
\draw[->,thick] (2,-1.2) -- (2,-1.8);
\draw[->,thick] (1.8,-1.8) -- (0.2,-1.2);
\draw (0,0) node {$s_2^{s_3}$};
\draw (0,-1) node {$s_1^{s_3}$};
\draw (2,-1) node {$s_5^{s_4}$};
\draw (2,-2) node {$s_3^{s_4}$};
\end{tikzpicture}\]
Note that
\begin{align*}
s_1^{s_3}&=1+(u_1-\phi_3(u_1)u_3) \otimes (\phi_1+\phi_1(u_3)\phi_3),\\
s_2^{s_3}&=1+(u_2-\phi_3(u_2)u_3) \otimes (\phi_2+\phi_2(u_3)\phi_3),\\
s_3^{s_4}&=1+(u_3-\phi_4(u_3)u_4) \otimes (\phi_3+\phi_3(u_4)\phi_4),\\
s_5^{s_4}&=1+(u_5-\phi_4(u_5)u_4) \otimes (\phi_5+\phi_5(u_4) \phi_4).
\end {align*}
Since
$\phi_5(u_2)=\phi_5(u_3)=\phi_4(u_2)=\phi_1(u_5)=\phi_1(u_4)=\phi_3(u_4)=0$,
we have
\begin{align*}
  w(s_1^{s_3},s_2^{s_3},s_5^{s_4})
  &=
     \phi_2(u_1)\Big(-\phi_5(u_4)\phi_3(u_2)\phi_4(u_3)\Big)\cdot
     \Big(\phi_1(u_3)\phi_3(u_5)\Big)\\
  &=-w(s_1,s_2,s_3) \cdot w(s_3,s_4,s_5) \not \in M,
\end{align*}
being $w(s_1,s_2,s_3) \not \in M$ and $0 \neq w(s_3,s_4,s_5) \in
M$. So we may assume that $(s_1^{s_3},s_2^{s_3},s_5^{s_4})$ is a
triangle of type 1.  This implies that it is a two-way directed cycle,
so $[s_1^{s_3},s_5^{s_4}]\in\E$.  Moreover, by \cref{lem:conjiso}
$[s_5^{s_4},s_3^{s_4}]$ is a single edge, being a conjugate of the
single edge $[s_5,s_3]$, and $[s_3^{s_4},s_1^{s_3}]\in\E$ since
\[
   (\phi_1+\phi_1(u_3)\phi_3)(u_3-\phi_4(u_3)u_4)=\phi_1(u_3)\neq 0
\]
being $\phi_1(u_4)=\phi_3(u_4)=0$. Thus,
$(s_1^{s_3},s_5^{s_4},s_3^{s_4})$ is a one-way directed triangle,
so $d_u(s_1^{s_3},s_5^{s_4},s_3^{s_4})\neq 0$ and we are done.

\medskip

%

\noindent\textbf{Case 2:} $[s_4,s_3] \in\E$.

Let $\lambda \in M$ (to be specified). 

Assume first that $w(s_2,s_1,s_4) \neq 0$. We may assume that
both triangles $(s_1,s_2,s_3)$ and $(s_1,s_3,s_2)$ are of type 1, by
\cref{rem:abc-acb}. The triangle
$(s_2,s_1,s_4)$ is of type $2$ (in which case we are done) unless
$(s_1,s_4)$ and $(s_4,s_2)$ are double edges (by \cref{rem:ds-du-4}), so we
may assume this is the case.

Since $w(s_2,s_4) \cdot d_u(s_5,s_3,s_4) \neq 0$,
there are at most two values of $\lambda\in M$ such that
$d_u(s_5,s_3,s_2^{s_4^{\lambda}})=0$. If
$w(s_3,s_1,s_4) \notin M$ then $(s_3,s_1,s_4)$ is a triangle of
type $1$ and we are done, so now assume that $w(s_3,s_1,s_4) \in
M$. Since $w(s_3,s_1,s_2) \notin M$, at least one of
$w(s_3,s_1,s_2^{s_4^{\lambda}})\notin M$ or
$w(s_3,s_1,s_2^{s_4^{-\lambda}})\notin M$ for any $\lambda\in M$. Since
$|M| \geq 5$, there exists $\lambda \in M$ such that
$d_u(s_5,s_3,s_2^{s_4^{\lambda}}) \neq 0$ and
$w(s_3,s_1,s_2^{s_4^{\lambda}}) \not \in M$. 
Thus, we are done in this case.

Assume now that $w(s_2,s_1,s_4)=0$. Since
$w(s_2,s_3) \cdot d_u(s_3,s_4,s_5) \neq 0$, there are at most two
values of $\lambda \in M$ such that
$d_u(s_2^{s_3^{-\lambda}},s_4,s_5)
=d_u(s_2,s_4^{s_3^{\lambda}},s_5^{s_3^{\lambda}}) = 0$.

Since $w(s_1,s_2,s_3) \not \in M$, by \cref{rem:abc-acb} we may assume
that $w(s_2,s_1,s_3) \not \in M$, so that
$w(s_3,s_4) \cdot w(s_2,s_1,s_3) \not \in M$ being
$w(s_3,s_4) \in M^\times$.
We would like to apply Lemma \ref{lem:af} in order to obtain that
$w(s_2,s_1,s_4^{s_3^{\lambda}}) \not \in M$,
while at the same time $d_u(s_2^{s_3^{-\lambda}},s_4,s_5) \ne 0$.  It
was proved in \cref{lem:af} that $w(s_2,s_1,s_4^{s_3^{\lambda}})$ can
be expressed as $A+B\lambda+C \lambda^2$. In our case we have
$A=w(s_2,s_1,s_4)=0$. If $\alpha,\beta\in M^\times,\ \alpha\neq \beta$
satisfies $B \alpha+C \alpha^2 \in M$ and $B \beta+C \beta^2 \in M$, then
Cramer's rule imply $B,C \in M$, which is not
the case since $C=w(s_3,s_4) \cdot w(s_2,s_1,s_3) \not\in M$.  Thus
$w(s_2,s_1,s_4^{s_3^{\lambda}})$ is in $M$ for at most one value of
$\lambda \in M^\times$.

Since $|M|\geq 5$, we obtain that there is a $\lambda\in M$ such that
$w(s_2,s_1,s_4^{s_3^{\lambda}})\notin M$ and
$d_u(s_2, s_4^{s_3^{\lambda}},s_5^{s_3^{\lambda}})\neq 0$. We are done.

\medskip

\noindent\textbf{Case 3:} $[s_3,s_5]\in\E$.

Let $\lambda \in M$ (to be specified).

Assume first that $w(s_1,s_2,s_5) \neq 0$. We may assume that both triangles
$(s_1,s_2,s_3)$ and $(s_1,s_3,s_2)$ are of type 1, by
\cref{rem:abc-acb}. The triangle $(s_1,s_2,s_5)$ is of type $2$ (in which
case we are done) unless $(s_2,s_5)$ and $(s_5,s_1)$ are double edges
(by \cref{rem:ds-du-4}), so we may assume this is the case.

Since $w(s_2,s_5) \cdot d_u(s_4,s_3,s_5) \neq 0$, there are at most
two values of $\lambda \in M$ such that
$d_u(s_4,s_3,s_2^{s_5^{\lambda}})=0$. If $w(s_1,s_5,s_3) \not \in M$
then $(s_1,s_5,s_3)$ is a triangle of type $1$ and we are done, so now
assume that $w(s_1,s_5,s_3) \in M$. Since
$w(s_1,s_2,s_3) \not \in M$, either
$w(s_1,s_2^{s_5^{\lambda}},s_3) \not \in M$ or
$w(s_1,s_2^{s_5^{-\lambda}},s_3) \not \in M$ for any $\lambda\in M$.
As in the previous case, since
$|M|\geq 5$, there exists $\lambda \in M$ such that
$d_u(s_4,s_3,s_2^{s_5^{\lambda}}) \neq 0$ and
$w(s_1,s_2^{s_5^{\lambda}},s_3) \not \in M$. 
Thus, we are done in this case.

Assume now that $w(s_1,s_2,s_5)=0$. Since
$w(s_2,s_3) \cdot d_u(s_3,s_5,s_4) \neq 0$, there are at most two
values of $\lambda \in M$ such that
$d_u(s_2^{s_3^{-\lambda}},s_5,s_4)=d_u(s_2,s_5^{s_3^{\lambda}},s_4^{s_3^{\lambda}})=
0$.

Since $w(s_1,s_2,s_3) \not \in M$ and $w(s_3,s_5)\in M^\times$, we have
$w(s_3,s_5) \cdot w(s_1,s_2,s_3) \not \in M$.

Now, we can use the same argument as in the previous case but applied to
$w(s_1,s_2,s_5^{s_3^\lambda})=B\lambda+C\lambda^2$ (where $C\notin M$)
in order to find a
$\lambda\in M$ such that both $w(s_1,s_2,s_5^{s_3^\lambda})\notin M$ and
$d_u(s_2,s_5^{s_3^{\lambda}},s_4^{s_3^{\lambda}})\neq 0$.

\medskip

\noindent\textbf{Case 4:} $[s_4,s_3] \notin\E$, $[s_3,s_5] \notin\E$
and at least one of $[s_2,s_4] \in\E$ and $[s_2,s_5] \in\E$.

Let $\lambda \in M$ (to be specified).

Since $[s_4,s_3] \notin\E$, $w(s_3,s_4)=0$, moreover
$w(s_1,s_2,s_3) \not \in M$, therefore, up to replacing $\lambda$
with $-\lambda$, we may assume that
$w(s_1,s_2,s_3^{s_4^{\lambda}}) \not \in M$, so that
$(s_1,s_2,s_3^{s_4^{\lambda}})$ is a triangle of type 1. Therefore,
it is two-way directed, so $[s_3^{s_4^\lambda},s_2]\in \E$. Since at
least one of $[s_2,s_4]$ and $[s_2,s_5]$ is an edge and
$|M|\geq 3$, we may choose $\lambda \in M$ such that
$\phi_5(u_2) \neq \pm \lambda \phi_5(u_4) \phi_4(u_2)$ (the $\pm$ sign
is needed because in this argument we have possibly replaced $\lambda$
with $-\lambda$). Therefore
\[
  \phi_5(u_2)+\lambda \phi_5(u_4)
  \phi_4(u_2) \neq 0,
  \textrm{ that is}, [s_2,s_5^{s_4^\lambda}]\in\E.
\]
Note that
$[s_5^{s_4^{\lambda}},s_3^{s_4^{\lambda}}]$ is a single edge being a
conjugate of the single edge $[s_5,s_3]$. By \cref{rem:ds-du-4}, the
triangle $(s_2,s_5^{s_4^{\lambda}},s_3^{s_4^{\lambda}})$ is of type 2.

\noindent\textbf{Case 5:} $[s_4,s_3] \notin\E$, $[s_3,s_5] \notin\E$ and
at least one of $[s_5,s_1]\in\E$ and $[s_4,s_1]\in\E$.

Let $\lambda \in M$ (to be specified).

Since $[s_3,s_5] \notin\E$, $w(s_3,s_5)=0$, moreover
$w(s_1,s_2,s_3) \not \in M$, therefore, up to replacing $\lambda$
with $-\lambda$, we may assume that
$w(s_1,s_2,s_3^{s_5^{\lambda}}) \not \in M$, so that
$(s_1,s_2,s_3^{s_5^{\lambda}})$ is a triangle of type 1. Therefore,
it is two-way directed, so $[s_1,s_3^{s_5^\lambda}]\in \E$. Since at
least one of $[s_5,s_1]$ and $[s_4,s_1]$ is an edge and
$|M| \geq 3$, we may choose $\lambda \in M$ such that
$\phi_1(u_4) \neq \pm \lambda \phi_5(u_4) \phi_1(u_5)$ (the $\pm$ sign
is needed because in this argument we have possibly replaced $\lambda$
with $-\lambda$). Therefore
\[
  \phi_1(u_4)-\lambda \phi_5(u_4) \phi_1(u_5)
  \neq 0,
  \textrm{ that is}, [s_4^{s_5^\lambda},s_1]\in\E.
\]
Note that $[s_3^{s_5^{\lambda}},s_4^{s_5^{\lambda}}]$ is a
single edge being a conjugate of the single edge $[s_3,s_4]$. By
\cref{rem:ds-du-4}, the triangle
$(s_1,s_3^{s_5^{\lambda}},s_4^{s_5^{\lambda}})$ is of type 2.

Now, we can finish the proof of our main theorem when $G=SL(V)$.
\begin{proof}[Proof of \cref{thm:main} for $G=SL(V)$]
  In the previous sections, we already showed how we can generate an
  $M$-closed set of transvections $X$.  The above discussion implies
  that, in the case $G=\SL(V)$, up to extending the $M$-closed $X$
  with exponent bounded by ${O(1)}$, we can reach
  $L_2(X)=\mathbb{F}_q$. Applying \cref{lem:M-closure-of-X-2}, we may
  assume that $X$ is $\mathbb{F}_q$-closed.  In particular, such an
  $X$ contains a full transvection group over $\FF q$ in length
  $(\log|G|)^{O(1)}$. Now, the theorem follows from \cite[Theorem
  1.5]{HZ}
\end{proof}
\subsection{Addition of parameters defining transvections}
\label{sec:addition}
In view of the previous section, from now on we assume that $V$ is a
non-degenerate symplectic or hermitian space with defining form $f$,
$G=\Sp(V)$ or $G=\SU(V)$
and $X\subset G$ is an $\FF{q_0}$-closed subset of generating transvections
with $\diam(\mt(X))\leq 2$. For simpler notation, we can add $1$ to $X$
and assume that $1\in X$. 

Recall that for any non-zero singular vector $v\in V$,
the transvection subgroup associated to $v$ is
$T_v=(1+v\otimes \varphi_v)^\mf$, where in the symplectic case
$\mf=\FF {q}$, while in the unitary case $\mf=\FF{q_0}\cdot \lambda_0$
where $\lambda_0\in\FF q^\times$ is any field element with
$\Tr(\lambda_0)=0$. In what follows, for any non-zero
singular $v\in V$, we denote by $t_v$ an arbitrary non-zero element of
the associated transvection subgroup $T_v$. 

By our assumptions, we have $X=\bigcup\,\{T_v\,|\, v\in {_VX}\}$.  Our
main goal is to prove that $\ell_X(T_{\sum_{i=1}^k a_i})= O(n^c)$ for
any $T_{a_1},\ldots,T_{a_k}\subset X$ where $k\leq n$ and
$\sum_{i=1}^ka_i$ is singular.  Since $\langle _VX\rangle=V$ this
implies that $\ell_X(\mt)=O(n^c)$. (Note that since
$T_{a_i}=T_{\lambda_ia_i}$, we do not need to take arbitrary linear
combinations.) First we prove this for $k=2$.
\begin{lem}\label{lem:t_a+b}
  Let $T_a,T_b\subset X$ for some $a,b\in V$ such that $a+b$ is singular.
  Then $\ell_X(T_{a+b})\leq c$ for some constant $c$.
\end{lem}
\begin{proof}
  If $\fl a\fr=\fl b\fr$, then $T_{a+b}=T_a=T_b$, so there is nothing
  to prove.  If $(t_a,t_b)$ is a (two-way) directed edge, then let
  $W=\fl a,b\fr$. Now, $V=W\oplus W^\perp$ where
  $W^\perp=\ker(\varphi_a)\cap \ker(\varphi_b)$, and the restriction
  to $W$ defines an isomorphism
  $\fl T_a,T_b\fr\rightarrow \fl (T_a)_W,(T_b)_W\fr \simeq \SL(2,K)$
	for some field $K \leq \FF q$. By \cref{thm:Dickson_for_subsets}, $K$
    must be equal to $\FF {q_0}$. Since
  $(T_{a+b})_W\leq \SL(2,\FF {q_0})$, we can apply the strong form of
  Babai's conjecture for $\SL(2,q_0)$ to get
  \[
    \ell_X(T_{a+b})\leq \ell_{T_a\cup T_b}(T_{a+b})=
    O\Big(\frac{\log|\SL(2,q_0)|}{\log|T_a\cup T_b|}\Big)=O(1).
  \]

  Now, let us assume that $(t_a,t_b)$ is not an edge.
  Since $X$ has diameter $2$, there exists a path $t_a,t_c,t_b$ in
  $\Gamma(X)$ for some $t_c\in X$. So $c\in\ _VX$ satisfies
  $f(c,a) \neq 0$, and $f(b,c) \neq 0$.

  Since $f(b,a)=0$, $f(a,a)=0$, $f(c,a) \neq 0$ we
  deduce that $c \not \in \fl a,b \fr$. This implies that
  $\langle a,b,c \rangle$ is a space of dimension $3$. Let
  \[w := -f(c,b)a+f(c,a)b.\] It is not hard to see that $\fl w \fr$ is
  the radical of the space $\fl a,b,c \fr$. Since $X_{V^*}$ generates
  $V^{*}$, there exists $t_d\in X$ with $f(d,w) \neq 0$. Therefore, at
  least one of $(t_a,t_d)$ and $(t_b,t_d)$ is an edge in $\Gamma(X)$.
  Let $W := \langle a,b,c,d \rangle \leq V$. We claim that the radical
  of $W$ is trivial. Indeed, $d \not \in \fl w \fr^{\perp}$ while
  $\fl a,b,c \fr$ is contained in $\fl w \fr^{\perp}$ so
  $\fl w \fr^{\perp}\cap W = \fl a,b,c \fr$. Thus the radical of $W$
  is contained in $\fl a,b,c \fr$ so it is contained in $\fl w
  \fr$. We conclude that the radical of $W$ is trivial since $w$ is
  not orthogonal to $d$.

  Let $X'=(T_a)_W\cup(T_b)_W\cup(T_c)_W\cup(T_d)_W$. We prove that it
  generates either $\Sp(W)\simeq \Sp(4,q)$ or $\SU(W)\simeq \SU(4,q)$.
  In order to do so it is enough to guarantee that the conditions of
  \cref{thm:Wagner} hold. Clearly, conditions (\ref{wagneritem1}) and
  (\ref{wagneritem3}) hold. The irreducibility of $X'$ will follow if
  we verify the conditions of \cref{thm:irred-condition}. Plainly,
  $a,b,c,d$ generates $W$ and the non-degeneracity of $W$ implies
  condition (2). The transvection graph induced by $X'$ is connected
  since $a$ and $b$ are not orthogonal to $c$ and $d$ is not
  orthogonal to at least one of $a$ and $b$. Thus the corresponding
  transvections in the transvection subgroups are connected and hence
  the induces subgraph of the transvection graph is strongly connected
  since every edge is two-way directed in this case.

  The bounded-rank
  case of the Babai conjecture (in its strong form) can be applied to
  deduce that
  \[
    \ell_X(T_{a+b})\leq \ell_{X'}((T_{a+b})_W)=
    O\Big(\frac{\log|\SL(4,q)|}{\log|X'|}\Big)=O(1).
  \]
\end{proof}
\begin{rem}\label{rem:symp-c=21}
  In the symplectic case, we managed to find elementary arguments to
  prove this Lemma with specific constant $c=21$ as follows.

  \textbf{Case 1, $\mathbf{(a,b)\in\E}$:} Let
  $T_a=(1+a\otimes \varphi_a)^{\FF q},\ T_b=(1+b\otimes
  \varphi_b)^{\FF q}$. Choosing $\lambda := \varphi_b(a)^{-1}$, we get
  $T_{a+b}=(1+\lambda b \otimes \varphi_b) (1+a \otimes
  \varphi_a)^{\FF q} (1-\lambda b \otimes \varphi_b)$, so
  $\ell_X(T_{a+b})\leq 3$.

  \textbf{Case 2, $\mathbf{(a,b)\notin\E}$:}
  Let $T_c,T_d\subset X$ as in the proof of \cref{lem:t_a+b}.

  First, let us assume that $f(c,a+b)\neq 0$. Since $(a,c)$ and
  $(b,c)$ are edges by construction, we have $\ell_X(T_{a+c})\leq 3$
  and $\ell_X(T_{b-c})\leq 3$ by Case 1. Then we have
  $f(a+c,b-c)=f(c,a+b)\neq 0$, so using Case 1 again we get that
  $\ell_{T_{a+c}\cup T_{b-c}}(T_{a+b})\leq 3$. Hence $\ell_X(T_{a+b})\leq 9$.

  Now, let us assume that $f(c,a+b)=0$. By construction,
  $f(d,a+b)\neq 0$, which implies that at least one of $f(d,a)$ and
  $f(d,b)$ is not zero. If both of them are not zero, then we can use
  the same argument as in the previous case (but using $d$ instead of
  $c$). So let us assume, say, $f(d,a)\neq 0$ and $f(d,b)=0$.

  Let $d'=d$ if $f(c,d)\neq 0$, while $d'=a+d$ if $f(c,d)=0$.
  By our assumptions and by Case 1,
  \[
    \ell_X(T_{d'})\leq 3,\quad (t_c,t_{d'}),
    (t_a,t_{d'}),(t_{a+b},t_{d'})\in \E\textrm{ but }(t_b,t_{d'})\notin\E.
  \]
  Using the construction of Case 1 again, we can deduce that
  $\ell_X(T_{\tau c+d'})\leq 5$ for every $\tau\in \FF q^\times$.  Since
  $q>2$, we may choose a $\tau\neq 0$ satisfying $f(\tau c+d',a)\neq 0$.
  Now, $t_{\tau c+d'}$ is a neighbour of each of $t_a,t_b,t_{a+b}$, so
  the argument of the first paragraph of Case 2 (but using
  $\tau c+d'$ instead of $c$) can be used to construct $t_{a+b}$. Using this
  construction in a careful way, one can show that $\ell_X(T_{a+b})\leq 21$.
\end{rem}
 Now, we are able to generate all transvections in the symplectic case.
\begin{thm}
  Let us assume that $G=\Sp(V)$ and  let $c$ be the constant as in
  \cref{lem:t_a+b}. Then we have $\ell_X(\mt)\leq cn^{\log_2(c)}$. In view of
  \cref{rem:symp-c=21}, the bound $\ell_X(\mt)\leq 21n^{4.4}$ holds.
\end{thm}
\begin{proof}
  Since $\mt=\bigcup_{0\neq v\in V}T_v$ it is enough to generate $T_v$
  for any $0\neq v\in V$. Our argument is essentially the same as the
  proof of \cite[Lemma 4.12]{HZ}. For the convenience of the reader,
  we present a detailed proof.

  Let $v\neq 0$ be a vector in $V$.  Since $ _VX$ generates
  $V$, there are $a_1,\ldots,a_k\in _V X$
  such that $k\leq n$ and
  $v=\sum_{i=1}^k a_i$. Let $l(k)=\lceil \log_2 k\rceil$ so
  $l(k)$ is the smallest integer satisfying $k\leq 2^{l(k)}$.  We
  prove that $\ell_X(T_v)\leq c^{l(k)}\leq c n^{\log_2(c)}$ by
  using induction on $k$.

  The claim trivially holds for $k=1$. For an arbitrary $k\leq n$, let
  \[v_1=\sum_{i=1}^{\lceil k/2\rceil} a_i\textrm{ and }
    v_2=\sum_{i=\lceil k/2\rceil+1}^k a_i.\]
  Let $X'=X\cup T_{v_1}\cup T_{v_2}$.
  Since $l(\lceil k/2\rceil)\leq l(k)-1$, by induction we have
  $\ell_{X}(X')\leq c^{l(k)-1}$. On the other hand, $\ell_{X'}(T_v)\leq c$
  by \cref{lem:t_a+b}. So,
  $\ell_X(T_v)\leq \ell_X(X')\cdot \ell_{X'}(T_v)\leq c^{l(k)}$, as claimed.
\end{proof}
For the remainder, let $V$ be a non-degenerate hermitian space over
$\FF q$ and  $G=\SU(V)$. Now, $q_0=\sqrt q$, so $\FF q$ is a
$2$-dimensional $\FF {q_0}$-space.
\begin{lem}\label{lem:t_a+b+c}
  Let $T_{v_1},T_{v_2},T_{v_3}\subset X$ for some $v_1,v_2,v_3\in V$
  and let us assume that $v_1+v_2+v_3$ is singular. Then
  $\ell_X(T_{v_1+v_2+v_3})\leq c'$ for some constant $c'$.
\end{lem}
\begin{proof}
  Let $\alpha_{ij}=f(v_i,v_j)$ for every $i\neq j$.

  First, let us assume that $\Tr(\alpha_{ij})=0$ for some $i\neq j$.
  If, for example, $\Tr(\alpha_{12})=0$, then
  $f(v_1+v_2,v_1+v_2)=\alpha_{12}+\alpha_{21}=\Tr(\alpha_{12})=0$, so
  $v_1+v_2$ is singular. Applying \cref{lem:t_a+b} twice, we get that
  $\ell_X(T_{v_1+v_2+v_3})\leq \ell_X(T_{v_1+v_2})\cdot \ell_{X\cup
    T_{v_1+v_2}}(T_{v_1+v_2+v_3})\leq c^2$.  Thus, for the remainder
  we assume that $\Tr(\alpha_{ij}) \ne 0$ for every $i\neq j$.  In
  particular, each $\alpha_{ij}\neq 0$ ($i\neq j$), that is,
  $t_{v_1}, t_{v_2}, t_{v_3}$ is a triangle in $\Gamma(X)$.

  Now, let us assume that $\alpha_{31}/\alpha_{21}\notin \FF {q_0}$.
  We claim that there is a $\lambda\in \FF q$ such that both
  $\lambda v_1+v_2$ and $(1-\lambda)v_1+v_3$ are singular.
  For any $\lambda\in \FF q$ we have
  \[
    \lambda v_1+v_2\textrm{ is singular }\iff \Tr(\lambda \alpha_{21})=
    f(\lambda v_1+v_2,\lambda v_1+v_2)=0.
  \]
  Similarly,
  \[
    (1-\lambda)v_1+v_3 \textrm{ is singular} \iff
    \Tr((1-\lambda)\alpha_{31})=0\iff \Tr(\lambda \alpha_{31})=\Tr(\alpha_{31}).
  \]
  For any $\gamma\in \FF q^\times$, the function
  $x\rightarrow \Tr(x\gamma)$ is an $\FF {q_0}$-linear map from
  $\FF q$ onto $\FF {q_0}$, so
  $\{x\in \FF q\,|\,\Tr(x\gamma)=\delta\}\subset \mathbb{F}_q$ is an
  affine line of the $\FF {q_0}$-space $\FF q$.  Thus, we need a
  $\lambda\in \FF q$, which is in the intersection of the affine lines
  \[
    \{x\in \FF q\,|\,\Tr(x\alpha_{21})=0\}\textrm{ and }
    \{x\in \FF q\,|\,\Tr(x\alpha_{31})=\Tr(\alpha_{31}).\}
  \]
  Our assumption $\alpha_{31}/\alpha_{21}\notin \FF{q_0}$ exactly
  means that these lines are not parallel, so there is a unique
  such $\lambda$. Using \cref{lem:t_a+b} again, we get that
  $\ell_X(T_{v_1+v_2+v_3})\leq \ell_X(T_{\lambda v_1+v_2})\cdot \ell_{X\cup
    T_{\lambda v_1+v_2}}(T_{v_1+v_2+v_3})\leq c^2$.
  Applying any permutation of the indeces we can also
  assume that $\alpha_{12}/\alpha_{32}$ and $\alpha_{23}/\alpha_{13}$ are
  in $\FF {q_0}$ for the remainder.

  Finally, let
  $k:=\alpha_{31}/\alpha_{21}=\alpha_{13}/\alpha_{12}\in \FF {q_0}^{\times}$,
  $l:=\alpha_{23}/\alpha_{13}=\alpha_{32}/\alpha_{31}\in \FF {q_0}^{\times}$,
  and $m:=\alpha_{12}/\alpha_{32}=\alpha_{21}/\alpha_{23}\in \FF {q_0}^{\times}$.
  Then we have
  \begin{align*}
    \alpha_{21}&=m\alpha_{23}=lm\alpha_{13}=klm\alpha_{12},\\
    \alpha_{13}&=\frac{\alpha_{12}}{\alpha_{21}}\alpha_{31}=
    \frac{\alpha_{31}}{klm},\\
    \alpha_{32}&=\frac{\alpha_{31}}{\alpha_{13}}\alpha_{23}=klm\alpha_{23}.
  \end{align*}
  Furthermore, $0\neq \Tr(\alpha_{12})=\alpha_{12}(1+klm)$, so $klm \neq -1$.
  Thus, we get
  \[
    \begin{vmatrix}
      0&\varphi_{v_1}(v_2)&\varphi_{v_1}(v_3)\\
      \varphi_{v_2}(v_1)&0&\varphi_{v_2}(v_3)\\
      \varphi_{v_3}(v_1)&\varphi_{v_3}(v_2)&0
    \end{vmatrix}=\alpha_{12}\alpha_{23}\alpha_{31}+
    \alpha_{21}\alpha_{32}\alpha_{13}=\alpha_{12}\alpha_{23}\alpha_{31}(1+klm)
    \neq 0.
  \]
  It follows that $W:=\fl v_1,v_2,v_3\fr$ is a non-degenerate
  $3$-dimensional subspace of $V$ and $V=W\oplus W^\perp$. Thus,
  $T_{v_i}=(T_{v_i})_W\oplus 1_{W^\perp}$ for each $i$, and
  $T_{v_1+v_2+v_3}=(T_{v_1+v_2+v_3})_W\oplus 1_{W^\perp}$. Defining
  $X'=(T_{v_1})_W\cup (T_{v_2})_W\cup(T_{v_3})_W\subset \SU(W)$ we
  have $\fl X'\fr=\SU(W)$ by \cref{thm:irred-condition} and
  \cref{thm:Wagner}, so we can apply the strong form of Babai's
  conjecture to deduce that
  \[
    \ell_X(T_{v_1+v_2+v_3})\leq \ell_{X'}((T_{v_1+v_2+v_3})_W)\leq
    O\Big(\frac{\log|\SU(3,q)|}{\log|X'|}\Big)=O(1).
  \]
  The proof is complete.
\end{proof}
\begin{lem}\label{lem:divide-to-singulars}
  Let $v=\sum_{i=1}^sv_i$ for some singular vectors
  $v_1,\ldots,v_s\in V$.  Then $v-\lambda v_i$ is singular for
  some $1\leq i\leq s$ and for some $\lambda\in \FF{q_0}$.
\end{lem}
\begin{proof}
  If $v$ is singular, then we can take $\lambda=0$. Otherwise,
  $0\neq 2f(v,v)=\sum_{i=1}^s \Tr(f(v,v_i))$, so there is an $i$ such
  that $\Tr(f(v,v_i))\neq 0$. Choosing
  $\lambda=\frac{f(v,v)}{\Tr(f(v,v_i))} \in \FF{q_0}$ we get that
  $f(v-\lambda v_i,v-\lambda v_i)=f(v,v)-\lambda f(v_i,v)- \lambda
  f(v,v_i)=f(v,v)-\lambda\Tr(f(v,v_i))=0$.
\end{proof}
\begin{thm}
  Let us assume that $G=\SU(V)$ and let $c$ be the constant as in
  \cref{lem:t_a+b+c}. Then we have $\ell_X(\mt)\leq O(n^{2\log_2(c)})$.
\end{thm}
\begin{proof}
  Let $v\in V$ be any non-zero singular vector. We need to generate $T_v$
  from $X$. As in the symplectic case, since ${_V}X$ generates $V$,
	there exist $a_1,\ldots,a_k \in {_V}X$ such that $k \leq n$ and
  $v=\sum_{i=1}^k  a_i$, $T_{a_i}\subset X$ for every $1\leq i\leq k$.
  Using the notation $l(k)=2 \log_2 k$, our goal is to prove that
  $\ell_X(T_v)\leq O(c^{l(k)})= O(n^{2\log_2(c)})$.

  Previously we proved this claim for $k\leq 3$. For an arbitrary
  $4\leq k\leq n$ our goal is to write $v$ as a sum of 3 singular
  vectors $v_1,v_2,v_3$ two of which are a linear combination of
  roughly $k/2$ many $a_i$'s.  First we consider the decomposition
  $v=u_1+u_2$ with
  \[u_1=\sum_{i=1}^{\lceil k/2\rceil}  a_i\textrm{ and }
    u_2=\sum_{i=\lceil k/2\rceil+1}^k  a_i.\]

  Using \cref{lem:divide-to-singulars} with, say, $i=1$ (which we can
  assume), we get that $v_1:=u_1-\lambda a_1$ is singular for some
  $\lambda\in \FF {q_0}$.  Now, if $(1+\lambda)a_1+u_2$ is singular,
  then we choose $v_2=(1+\lambda)a_1+u_2,\,v_3=0$. Finally, if
  $(1+\lambda)a_1+u_2=(1+\lambda)a_1+\sum_{i=\lceil k/2\rceil+1}^k
  a_i$ is not singular, then we apply \cref{lem:divide-to-singulars}
  again to write $(1+\lambda)a_1+u_2=v_2+v_3$ with singular vectors
  $v_2=\mu a_s$, $v_3=(1+\lambda)a_1+u_2-\mu a_s$ for some
  $s\in\{1,\lceil k/2\rceil+1,\ldots,k\}$ and for some
  $\mu \in \FF {q_0}$.

  Let $X'=X\cup T_{v_1}\cup T_{v_2}\cup T_{v_3}$. By our construction,
  each $v_i$ is a linear combination of at most $\lceil k/2\rceil+1<k$
  many $a_i$'s. Now, if $k$ is bounded, then $\ell_X(T_v)=O(1)$ follows by
  a repeated application of \cref{lem:t_a+b+c}.

  So, we can assume that $k\geq 10> 4/(\sqrt 2-1)$. Then
  $\lceil k/2\rceil+1\leq k/2+2\leq k/\sqrt 2$,
  that is, $l(\lceil k/2\rceil+1)\leq l( k/\sqrt 2)=2\log_2(k/\sqrt 2)=l(k)-1$.
  Using an induction argument and \cref{lem:t_a+b+c}, we get that
  $\ell_X(T_v)\leq \ell_{X}(X')\cdot \ell_{X'}(T_v)\leq O(c^{l(k)-1})\cdot c
  =O(c^{l(k)})$. The result follows.
\end{proof}
In order to obtain our main result it is enough to see that
$l_{\mt}(G)=O(\frac{\log |G|}{\log |\mt|})$. Since $\mt$ is a
conjugacy class, one can see that this holds by a result of Liebeck
and Shalev \cite{Liebeck-Shalev}. On the other hand
$l_{\mt}(G)=O(n^2)$ can easily be proved using explicit Gaussian
elimination-like algorithm.
\subsection{Proof of  \cref{thm:X-contains-K-tr-sgrp}}
The proof of  \cref{thm:X-contains-K-tr-sgrp} follows easily from our
previous proof. We only give a sketch here.

The only new thing we need is the following modification of
\cref{thm:Dickson_for_subsets} in our situation. 
\begin{prop}\label{prop:Dickson-with-K0-tr-sgrp}
  With the assumptions of \cref{thm:X-contains-K-tr-sgrp},
  let $(s,t)\in \E$ and $W=\fl s,t\fr$. Then
  $H:=\fl s^{K_0},t\fr\simeq Sp(W)\oplus 1_{W^\perp}$ or
  $SU(W)\oplus 1_{W^\perp}$, moreover
  $\diam(\Cay(H,\{s^{K_0},t\}))=O(1)$.
\end{prop}
Now, in order to prove \cref{thm:X-contains-K-tr-sgrp}, first we can
apply a modification of the argument of \cref{lem:Xconj} to get a
generating set of transvections $X$ for $G$ (containing a
transvections subgroup over $K_0$), in length $O(n^2)$ (for details,
see the second paragraph of the proof of \cite[Lemma 4.2]{HZ}). After
that, we can use the arguments of \cref{sec:dec-diameter}, to enlarge
$X$ in order to guarantee that the diameter of $\Gamma(X')$ is $2$ for
every $X\subset X'\subset \mt$.  Now, we need only use the first
paragraph of the proof of \cref{lem:M-closure-of-X} along with
\cref{prop:Dickson-with-K0-tr-sgrp}, to generate the $K_0$-closure of
$X$, even in length $O(1)$. After that, we can use (a modification of)
the arguments of \cref{sec:addition} to get all the transvections.

\end{document}